\documentclass[11pt,a4paper,table]{article}
\usepackage[applemac]{inputenc}
\usepackage[english]{babel}
\usepackage{amsmath, amsthm, amssymb}
\usepackage{geometry}
\usepackage{graphicx}
\usepackage{wrapfig}
\usepackage{placeins}
\usepackage{array}
\usepackage{tabularx}
\usepackage{hhline}
\usepackage{forest}
\usepackage{tikz} 
\usepgflibrary{decorations.pathmorphing}
\usetikzlibrary{arrows,backgrounds,fit,shapes.geometric,shapes,shapes.callouts,decorations.text,intersections,shadows,calc,through,decorations.pathmorphing,decorations}
\usepackage{marginnote}
\usepackage{subcaption}
\usepackage{fixltx2e}
\usepackage{dsfont}
\usepackage{pifont}
\usepackage{afterpage}

\renewcommand\appendix{\par
\setcounter{section}{0}%
\setcounter{subsection}{0}%
\setcounter{table}{0}
\setcounter{figure}{0}
\gdef\thetable{\Alph{table}}
\gdef\thefigure{\Alph{figure}}
\section*{Appendix}
\gdef\thesection{\Alph{section}}
\setcounter{section}{0}}

\newcommand{\GW}{\mathbb{P}}

\newcommand{\mP}{\mathbb{P}}
\newcommand{\oexp}{\mathrm{oe}}

\theoremstyle{plain}\newtheorem{teo}{Theorem}[section]
\newtheorem{prop}[teo]{Proposition}
\newtheorem{cor}[teo]{Corollary}
\newtheorem{lem}[teo]{Lemma}

\theoremstyle{remark}
\theoremstyle{definition}\newtheorem{df}{Definition}[section]
\renewcommand{\phi}{\varphi}
\renewcommand{\epsilon}{\varepsilon}

\newcommand{\1}{\mathds{1}}
\newcommand{\di}{\Delta}
\renewcommand{\root}{\emptyset}

\usepackage{hyperref}
\usepackage{pxfonts}

\title{\huge\textsc{The Scaling Limit\\of Random Outerplanar Maps}}
\author{\Large{Alessandra Caraceni}\footnote{Scuola Normale Superiore di Pisa, e-mail: alessandra.caraceni@sns.it}}
\date{}
\begin{document}
\maketitle

\begin{abstract}
A planar map is outerplanar if all its vertices belong to the same face. We show that random uniform outerplanar maps with $n$ vertices suitably rescaled by a factor $1/ \sqrt{n}$ converge in the Gromov-Hausdorff sense to ${7 \sqrt{2}}/{9}$ times Aldous' Brownian tree. The proof uses the bijection of Bonichon, Gavoille and Hanusse \cite{bijectionpaper}.
\end{abstract}

\section{Introduction}\label{intro}

\begin{figure}
\centering
\includegraphics[width=\textwidth]{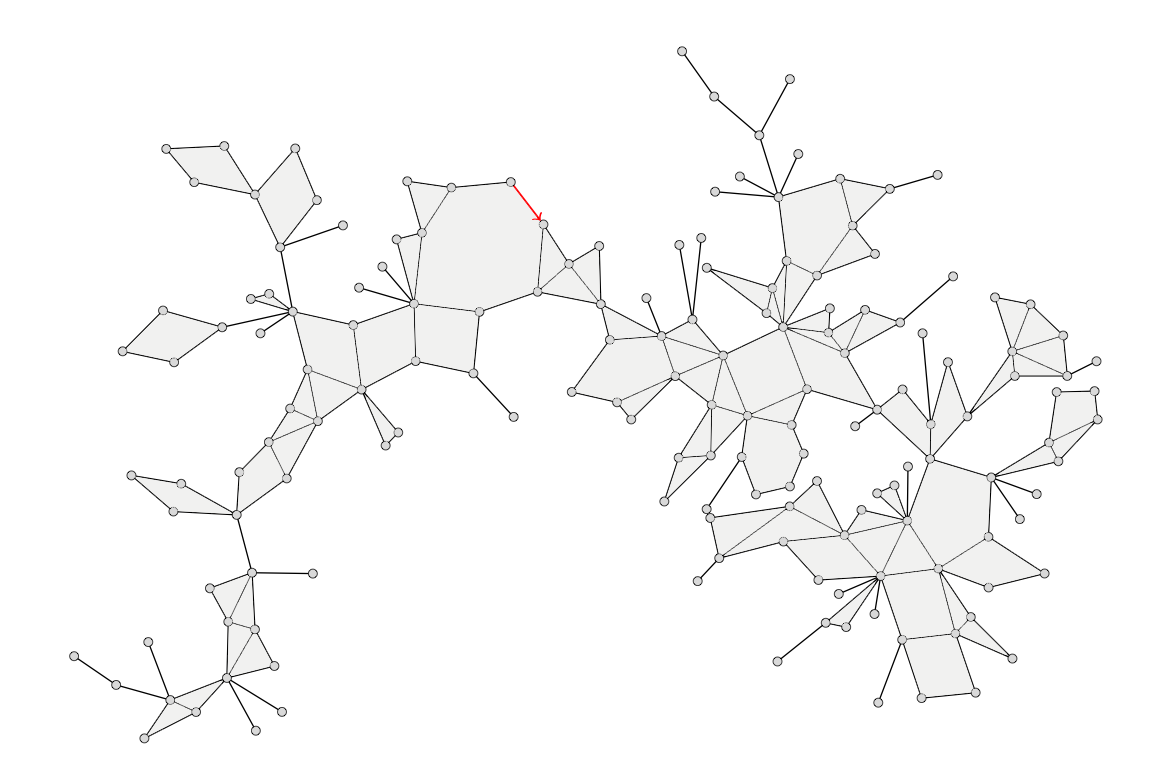}
\caption{A rooted, simple outerplanar map.}\label{outerplanar map}
\end{figure}

Since the early $90'$s a lot of effort has gone into the study of scaling limits for large combinatorial structures. The most emblematic result is, of course, the construction of the continuum random tree (CRT) by Aldous \cite{Ald91a,Ald91,Ald93} as the scaling limit of various classes of oriented trees. The CRT has since been shown to be a universal object: see e.g.\,\,\cite{ABBG12,AM08,CHKdissections,HM12,GW} and references therein, as well as the recent work \cite{PSW}. In this work we shall establish that the CRT is also the scaling limit of uniform random large outerplanar maps.

Recall that a planar map is a proper embedding of a finite connected graph into the plane (or the sphere), considered up to continuous deformations. A recent breakthrough was achieved by Le Gall and Miermont \cite{LG11,Mie11}, who showed that several classes of random planar maps admit the so-called Brownian map as scaling limit. It has been observed, however, that -- for some particular regimes -- random planar maps with a unique macroscopic large face have a tree-like structure and admit the CRT as scaling limit; see \cite{Bet11,BG09,JS12,PSW}. Our main result consists in a confirmation of this phenomenon for the case of outerplanar maps.

A map is \emph{outerplanar} if all of its vertices are adjacent to the same face, which is dubbed the \emph{outerface} and usually drawn as the infinite face in a planar embedding. Outerplanar maps constitute a well-studied combinatorial structure; in particular, they have a simple characterisation in terms of minors (a graph is outerplanar if and only if it does not contain $K_{2,3}$ nor $K_{4}$ as a minor \cite{minors}). See \cite{char} for more characterisations of outerplanar graphs. In this work we shall restrict ourself to \emph{simple} outerplanar maps, with no loops or multiple edges. As usual all of the maps considered here are \emph{rooted}, that is endowed with a distinguished oriented edge such that the outerface is lying on its left. The tail of the root edge will be called the \emph{root vertex}. Our main result is the following:

\begin{teo}\label{main}Let $\mathbf{M}_{n}$ be a random uniform rooted simple outerplanar map with $n$ vertices, and denote by $ d_{gr}$ the graph distance on  the set of its vertices $V(\mathbf{M}_{n})$. We have the following convergence in distribution for the Gromov-Hausdorff topology:
$$ \left( V(\mathbf{M}_{n}), \frac{d_{gr}}{ \sqrt{n}} \right) \quad \xrightarrow[n\to\infty]{(d)} \quad \frac{7 \sqrt{2}}{9} \cdot ( \mathcal{T}_{e},d),$$
where $ ( \mathcal{T}_{e}, d)$ is the Brownian CRT of Aldous. We adopt here the normalisation of Le Gall \cite{GW} by considering  $ \mathcal{T}_{e}$ as constructed from a normalised Brownian excursion.
\end{teo}

The first ingredient in our proof is a way to relate outerplanar maps to plane trees; this will be done using the bijection of Bonichon, Gavoille and Hanusse \cite{bijectionpaper} between the set of (simple and rooted) outerplanar maps with $n$ vertices and a special class of bicoloured plane trees with $n$ vertices which is described in Section~\ref{bijection}. The plan of the proof then partially follows that of \cite{CHKdissections}, in which Curien, Haas and Kortchemski prove the convergence of random dissections to a scalar multiple of the CRT. More specifically, we will show that the distances on an outerplanar map are roughly proportional to the distances on the associated tree. To this end, we describe throughout Section \ref{algorithm} an algorithm that, given a bicoloured tree and a vertex $v$, yields the length of a geodesic path from $v$ to the root vertex in the associated outerplanar map. When applied to the model of a bicoloured Galton-Watson tree conditioned to survive (presented in Section~\ref{infinite tree}) this algorithm yields a Markov chain whose mean increment (under the stationary distribution) gives the asymptotical proportionality constant between the metric on a large outerplanar map and that of its associated tree. The distances between arbitrary pairs of points are finally controlled by a large deviations estimate, see Sections \ref{markov} and \ref{conclusion}.

\section{Outerplanar maps and plane trees}\label{bijection}

As mentioned earlier, the first ingredient needed for our discussion is a bijection found in \cite{bijectionpaper}, which enables the coding of outerplanar maps as bicoloured trees of a certain class. More specifically,

\begin{df}We say that a rooted plane tree $\tau$ is \emph{bicoloured} if each of its vertices is coloured either black or white; we shall say that $\tau$ is \emph{well bicoloured} if it is bicoloured and all of the vertices in its rightmost branch are coloured white (see Figure \ref{bicoloured tree}). We shall henceforth simply write \emph{well bicoloured tree} when referring to a well bicoloured rooted plane tree.
\end{df}

\bigskip

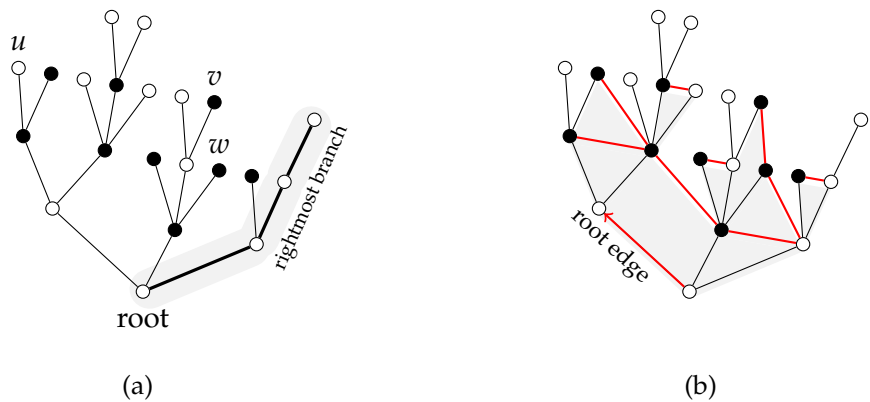
\begin{figure}[h]
\begin{subfigure}[b]{.5\textwidth}
\begin{tikzpicture}[node distance=20 pt, b/.style={circle,inner sep=0pt,minimum size=5pt, draw=black, fill=black}, w/.style={circle,inner sep=0pt,minimum size=5pt, draw=black, fill=white}, g/.style={circle,inner sep=0pt,minimum size=5pt, draw=black, fill=gray, parent anchor=north, child anchor=south}, r/.style={densely dashed,red}]
\begin{forest} for tree={l=25pt, grow'=80}
[,w, name=x, label=below:root[,w [,b,[,w,label=above:$u$][,b,,name=l]][,b,[,w][,b,[,w][,w]][,w]]][,b, [,b][,w[,w][,b,tikz={\draw[r, thick, ->] ()--(!un);},label=above:$v$]][,b,label=above:$w$]][,phantom][,w,name=y,edge=very thick,tikz={\begin{pgfonlayer}{background}\draw[gray!10, line width=15pt, line cap=round] ()--(!u);\end{pgfonlayer}}[,b][,w, edge=very thick,tikz={\begin{pgfonlayer}{background}\draw[gray!10, line width=15pt, line cap=round] ()--(!u) node[very near start, sloped,below, color=black, scale=.7] {\quad rightmost branch};\end{pgfonlayer}}[,phantom][,w,edge=very thick,tikz={\begin{pgfonlayer}{background}\draw[gray!10, line width=15pt, line cap=round] ()--(!u);\end{pgfonlayer}}]]]]
\end{forest}
\end{tikzpicture}
\vspace{-10pt}
\caption{}\label{bicoloured tree}
\end{subfigure} 
\begin{subfigure}[b]{.5\textwidth}
\begin{tikzpicture}[node distance=20 pt, b/.style={circle,inner sep=0pt,minimum size=5pt, draw=black, fill=black}, w/.style={circle,inner sep=0pt,minimum size=5pt, draw=black, fill=white}, g/.style={circle,inner sep=0pt,minimum size=5pt, draw=black, fill=gray, parent anchor=north, child anchor=south}, r/.style={thick,red}]
\begin{forest} for tree={l=25pt, grow'=80}
[,w, name=x, label=below:\phantom{root}[,w, tikz={\draw[r,<-] ()--(!u) node[near start, sloped,below, color=black, scale=.8] {root edge};\begin{pgfonlayer}{background}\fill[gray!10] (.south)--(!l.south)--(!s.south)--(!u.south)--(.south);\end{pgfonlayer}} [,b,tikz={\draw[r] ()--(!ul);\begin{pgfonlayer}{background}\fill[gray!10] (.south)--(!u.south)--(!s.south)--(.south);\end{pgfonlayer}}[,w][,b,,name=l,tikz={\draw[r] ()--(!uul);\begin{pgfonlayer}{background}\fill[gray!10] (.south)--(!u.south)--(!us.south)--(.south);\end{pgfonlayer}}]][,b,tikz={\draw[r] ()--(!un);}[,w,name=y,][,b,tikz={\draw[r] ()--(!n);\begin{pgfonlayer}{background}\fill[gray!10] (.south)--(!u.south)--(!ul.south)--(.south);\end{pgfonlayer}}[,w][,w]][,w]]][,b,tikz={\draw[r] ()--(!nn);\begin{pgfonlayer}{background}\fill[gray!10] (.south)--(!u.south)--(!ul.south)--(.south);\end{pgfonlayer}}[,b,tikz={\draw[r] ()--(!n);\begin{pgfonlayer}{background}\fill[gray!10] (.south)--(!u.south)--(!s.south)--(.south);\end{pgfonlayer}}][,w[,w][,b,tikz={\draw[r] ()--(!un);\begin{pgfonlayer}{background}\fill[gray!10] (.south)--(!u.south)--(!uu.south)--(!uul.south)--(.south);\end{pgfonlayer}}]][,b,tikz={\draw[r] ()--(!unn);\begin{pgfonlayer}{background}\fill[gray!10] (.south)--(!u.south)--(!uul.south)--(.south);\end{pgfonlayer}}]][,phantom][,w,name=z[,b,tikz={\draw[r] ()--(!n);\begin{pgfonlayer}{background}\fill[gray!10] (.south)--(!u.south)--(!ul.south)--(.south);\end{pgfonlayer}}][,w,[,phantom][,w]]]]
\end{forest}
\end{tikzpicture} 
\vspace{-10pt}
\caption{}\label{bijection example}
\end{subfigure}
\caption{Figure (a) shows a well bicoloured tree. Vertices $u$ and $v$ are \emph{unrelated}. Vertex $w$ is the \emph{target} of $v$. Figure (b) represents its image via the map $\Psi$.}
\end{figure}

Some working knowledge of the explicit bijection is needed in the sections that follow, and thus part of the construction is included for future reference. We adopt notation coherent with that of \cite{bijectionpaper}: given two distinct vertices in a plane tree, we call them \emph{unrelated} if neither is an ancestor of the other.

Let now $\tau$ be a well bicoloured tree. For each black vertex $v$  of $\tau$ define the \emph{target} of $v$ to be its next unrelated vertex in a clockwise contour of $\tau$. Define $\Psi(\tau)$ as the rooted outerplanar map obtained by joining each black vertex of $\tau$ to its target (via an edge that leaves the rightmost corner of the black vertex and enters the target from the leftmost corner available); root the map on the edge joining the former root of the tree to its leftmost child, oriented in such a way that the former root is the tail, and forget the colouring of $\tau$ (see Figure \ref{bijection example} for an example).

Notice that the order in which the additional edges are drawn is not relevant, and that the root edge will have the infinite face on its left side. An inverse of $\Psi$ can be constructed explicitly, but we refer the reader to \cite{bijectionpaper} for the details and proof, since all that we will need is the following:

\begin{teo}[Bonichon, Gavoille, Hanusse \cite{bijectionpaper}]\label{bij}The map $\Psi$ is a bijection between well bicoloured trees with $n$ vertices and simple rooted outerplanar maps with $n$ vertices.
\end{teo}

Such a bijection, together with the fact that the scaling limit for plane trees is the CRT, constitutes the basis for our future discussion; notice that, however, it is not at all clear how the colouring affects the metric in the switch from tree to map: distances between corresponding vertices are, in general, smaller when computed on the map (if two vertices are adjacent in the tree then they are in the corresponding map, but not vice-versa). This means we cannot easily employ the result for plane trees to make deductions on outerplanar maps. Most of the following sections will develop ways to control the outerplanar map metric via easily readable information about its corresponding tree.

\section{Rough localisation of geodesics}\label{preliminary lemmas}

We delve now into the central problem of the rather unclear relationship between distances on an outerplanar map and distances on its corresponding plane tree: given a well bicoloured tree $\tau$, we wish to compute distances on the map $\Psi(\tau)$. We restrict ourselves, in this section and many of the subsequent ones, to distances from the root vertex; in Section \ref{conclusion}, before the proof of Theorem 1.1, we will give a way to bound distances between arbitrary vertices of an outerplanar map with a function of distances from the root. In what follows, geodesics to the root in $\Psi(\tau)$ are built in a step-by-step manner, using local infomation about the tree structure of $\tau$ and its colouring.

\medskip

We shall refer to geodesics of $\Psi(\tau)$ as map-geodesics; since $\tau$ and $\Psi(\tau)$ have the same vertex set, any path in $\Psi(\tau)$ can be interpreted as a sequence of vertices $v_0\ldots v_n$ of $\tau$, where for each $i$ between 0 and $n-1$ vertices $v_i$ and $v_{i+1}$ are either neighbours in $\tau$ (parent and child, in any order), or a black vertex and its target (again, a priori, in any order).

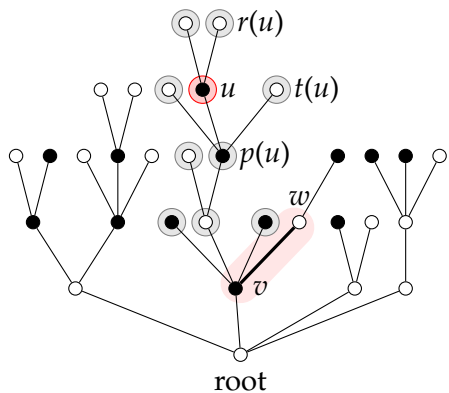
\begin{figure}
\centering
\begin{tikzpicture}[node distance=20 pt, b/.style={circle,inner sep=0pt,minimum size=5pt, draw=black, fill=black}, w/.style={circle,inner sep=0pt,minimum size=5pt, draw=black, fill=white}, g/.style={circle,inner sep=0pt,minimum size=5pt, draw=black, fill=gray, parent anchor=north, child anchor=south}, r/.style={densely dashed,red}]
\begin{forest} for tree={l=25pt, grow'=90}
[,w, name=x, label=below:root, use as bounding box[,w [,b,[,w,][,b,,name=l]][,b,[,w][,b,[,w][,w]][,w]]][,b,label=right:$v$[,b,tikz={\begin{pgfonlayer}{background}\draw[red!10, line width=15pt, line cap=round] (!u)--(!ul);\end{pgfonlayer}}][,w[,w][,b,label=right:$p(u)$[,w,tikz={\begin{pgfonlayer}{background}\node[draw=red, fill=red!20, circle] at (!n) {};\node[draw=gray, fill=gray!20, circle] at () {};\node[draw=gray, fill=gray!20, circle] at (!ul) {};\node[draw=gray, fill=gray!20, circle] at (!nl) {};\node[draw=gray, fill=gray!20, circle] at (!u) {};\node[draw=gray, fill=gray!20, circle] at (!uu) {};\node[draw=gray, fill=gray!20, circle] at (!uuu1) {};\node[draw=gray, fill=gray!20, circle] at (!uuu3) {};\node[draw=gray, fill=gray!20, circle] at (!uu1) {};\node[draw=gray, fill=gray!20, circle] at (!n1) {};\end{pgfonlayer}}][,b,label=right:$u$[,w][,w,label=right:$r(u)$,]][,phantom][,w,label=right:$t(u)$]]][,b][,w,label=above:$w$,edge=very thick,[,phantom][,phantom][,b]]][,phantom][,w,name=y[,b][,w]][,w[,w[,b][,b][,w]]]]
\end{forest}
\useasboundingbox (0,5);
\end{tikzpicture}
\vspace{10pt}
\caption{$(v,w)$ is a separating pair for $u$ (also, it is a separating pair for $r(u)$, $t(u)$, $p(u)$, or indeed for any of the circled vertices).}\label{separating pair}
\end{figure}

We need some additional notation: for each vertex $u$ of $\tau$, if $u$ is not the root, we write $p(u)$ for its parent; if $u$ has children in $\tau$, we call $r(u)$ its rightmost child; finally, if $u$ is a black vertex, we write $t(u)$ for its target. Also, given three vertices $u$, $v$ and $w$ of $\tau$, we say that $(v,w)$ is a \emph{separating pair} for $u$ (from the root of the tree) if $v$ is a strict ancestor of $u$ and $w$ is a child of $v$ lying to the right of $u$; see Figure \ref{separating pair}.

We write $d_M(u,v)$ for the graph distance between vertices $u$ and $v$ \emph{in the map} $\Psi(\tau)$, and simply write $d_M(u)$ for $d_M(u,\root)$, where $\root$ is the root of $\tau$. We are ready to prove the following:

\begin{prop}\label{vlemma}
In a well bicoloured tree $\tau$, let $(v,w)$ be a separating pair for $u$. Then any map-geodesic from $u$ to the root of $\tau$ will pass through $v$ or $w$ (possibly both). Consequently, any map-geodesic from $u$ to the root can be constructed by starting from $u$ and iteratively applying one of the maps $t$, $r$ or $p$, so that $r$ or $p$ are applied to white vertices, and $p$ or $t$ are applied to black vertices.
\end{prop}
\noindent{\it Proof.} Let $S$ be the set of (strict) descendants of $v$ that lie strictly to the left of $w$, and $u_0 u_1 \ldots u_n$ a map-geodesic such that $u=u_0$ and $u_n=\root$ is the root of the tree. Clearly, $u\in S$ and $u_n \notin S$; take the minimum $i$ such that $u_i\notin S$, and consider its relation to $u_{i-1}$. Children of elements in $S$ are in $S$, and for any pair $(x,t(x))$ in the tree (where $x$ is a black vertex), either both vertices belong to $S$ or neither does, except for the case where $x \in S$ and $t(x)=w$. Hence either $u_i$ is the parent of $u_{i-1}$ (therefore $u_i=v$) or $u_i$ is the target of $u_{i-1}$, which can only be the case if $u_i=w$, and thus is the first part of the proposition established.

\begin{wrapfigure}{R}{5cm}
\hspace{-55pt}
\begin{tikzpicture}[node distance=20 pt, b/.style={circle,inner sep=0pt,minimum size=5pt, draw=black, fill=black}, w/.style={circle,inner sep=0pt,minimum size=5pt, draw=black, fill=white}, g/.style={circle,inner sep=0pt,minimum size=5pt, draw=black, fill=gray, parent anchor=north, child anchor=south}, r/.style={densely dashed,red}]

\begin{forest} for tree={l=10pt, grow'=90}
[,b, label=left:$v$, name=x[,w [,b,tikz={\draw[r] ()--(!ul);}[,w][,b,,name=l,tikz={\draw[r] ()--(!uul);}]][,b,tikz={\draw[red,very thick] ()--(!un);}[,w,name=y,label=above:$u$,edge=very thick][,b,tikz={\draw[r] ()--(!n);}[,w][,w]][,w]]][,b,tikz={\draw[red,very thick] ()--(!nn);} [,b,tikz={\draw[r] ()--(!n);}][,w[,w][,b,tikz={\draw[r] ()--(!un);}]][,b,tikz={\draw[r] ()--(!unn);}]][,phantom][,w,label=below:$w$,name=z[,phantom][,b,edge=very thick,tikz={\draw[red,very thick,->] ()--+(10pt,-40pt);}]]]
\end{forest}
\begin{pgfonlayer}{background}
\node (R) [fill=gray!10,rounded corners=30pt,yshift=25pt,xshift=48pt, fit=(y)(z)(x)(l), use as bounding box]{};
\node[fill=red!10,circle,inner sep=2pt, yshift=6pt,xshift=65pt, fit=(y)]{};
\draw[red!10, line width=12pt, line cap=round] (x)+(66pt,6pt)--(105pt,25pt);
\node (dots) at (66pt,-10pt) {$\ldots$};
\draw (x)+(66pt,6pt)--(dots);
\draw[r] (x)+(66pt,6pt)--(0pt,16pt);
\end{pgfonlayer}
\node [draw=gray!10, very thick, right of=R,xshift=-96pt,yshift=24pt, circle, fill=white, inner sep=3pt] {$S$};
\end{tikzpicture} 
\end{wrapfigure}
Let us consider what this implies in term of distances: if $(v,w)$ is a separating pair for $u$, then $d_M(u)>d_M(v)$ or $d_M(u)>d_M(w)$; that is, since $|d_M(v)-d_M(w)|\leq 1$, $d_M(u)\geq \max\{d_M(v),d_M(w)\}$. Now, consider a map-geodesic from $u$ to the root. If $x$ is a child of $u$ distinct from $r(u)$, then such a geodesic does not go through $x$, because $(u,r(u))$ is a separating pair for $x$, and thus $d_M(x)\geq d_M(u)$ (whereas, if there were a geodesic from $u$ to the root that involved $x$, we would have $d_M(x)<d_M(u)$). If $u$ is the target of some $y$, then consider $p(u)$, which must be an ancestor of $y$: $(p(u),u)$ is a separating pair for $y$, and thus $d_M(y)\geq d_M(u)$, so a map-geodesic from $u$ to the root does not go through $y$. Only three possibilities remain: either the geodesic moves from $u$ to $p(u)$, or to $t(u)$, or to $r(u)$.

Clearly, a white vertex $u$ in a map-geodesic to the root will be followed by either $r(u)$ or $p(u)$, since it is not directly connected to any target in the map. Suppose, on the contrary, that $u$ is a black vertex; it will be followed in a map-geodesic to the root by $p(u)$ or $t(u)$, never by a child: this is because $(p(t(u)),t(u))$ is a separating pair for $u$; if the map-geodesic does not move from $u$ to $t(u)$, then it passes through $p(t(u))$, which has map-distance at most 2 from $u$, and at least 2 from any child of $u$ (it cannot be the target of one, since children of $u$ have $t(u)$ or other children of $u$ as targets, and is not connected to strict descendants of $u$ in the tree); as a consequence, the map-geodesic does not go through $r(u)$.\qed

\section{An algorithm to compute map-distances}\label{algorithm}

Suppose we have a well bicoloured tree $\tau$ and a vertex $x$ of $\tau$ of height $n$; we know that $d_M(x)$, the map-distance between $x$ and the root of $\tau$, is no more than $n$. From now on, we write $d(\tau, x)$ for the map distance $d_M(x)$ between $x$ and the root of $\tau$. We propose to compute $d(\tau,x)$ via a recursive algorithm which takes the pair $(\tau, x)$ as input and outputs a pair $(\tau', x')$, where $x'$ is a vertex of $\tau'$ such that $d(\tau',x')=d(\tau,x)-1$; this way the number of iterations needed for the algorithm to output a well bicoloured tree and its root is exactly the length of a geodesic path from $x$ to the root of $\tau$.

Given $(\tau,x)$, consider the path \emph{in the tree} leading from $x$ to the root; thanks to Proposition~\ref{vlemma} we know that a map-geodesic from $x$ to the root cannot involve any of the vertices that lie strictly to the left of this path (parents, targets and rightmost children of vertices that are part of the path or lie to the right of it cannot lie to its left). We may thus safely erase all such vertices from $\tau$, and we will always output pairs $(\tau', x')$ such that no vertices lie strictly to the left of the tree path leading from $x'$ to the root of $\tau'$.

In what follows, given a tree $\tau$ and a vertex $x$, we will write $\tau^x$ for the subtree of $\tau$ formed by $x$ and its descendants; given a tree $\tau$ and a subtree $\tilde{\tau}$ (which does not include the root of $\tau$), we will write $\tau\setminus \tilde{\tau}$ for the rooted plane tree obtained from $\tau$ by erasing all vertices of $\tilde{\tau}$ and any edges adjacent to those vertices.

The description of the algorithm follows.
\\\\
\begin{tabularx}{\textwidth}{p{.7cm}X}
{\tikz[baseline=10pt]
{\path[use as bounding box] (-12pt,-15pt) rectangle (12pt,20pt);
\node[circle,inner sep=0pt,minimum size=4pt,fill=white,draw=black,label=above:$x$] (1) {};
\node[above of=1,yshift=-20pt] (2) {};
\node[circle,inner sep=0pt,minimum size=4pt,fill=gray,draw=black,label=left:$p(x)$,very thick,yshift=-15pt] (3) {};
\node[yshift=-25pt] (d) {$\ldots$};\draw (1)-- (3)--(d);
\begin{pgfonlayer}{background}
\node[ellipse, fit=(1)(2), inner sep=1pt,fill=red!10, draw=black, densely dashed] (ellipse) {\phantom{$T^x$}};
\node at (ellipse.north) {\raisebox{3pt}{\ding{33}}};
\end{pgfonlayer}
}
}
 &
\textbf{[w\textsubscript{0}]} Suppose $x$ is a white leaf of $\tau$; then a map-geodesic to the root necessarily moves from $x$ to $p(x)$ (hence $d(\tau,p(x))=d(\tau,x)-1$). Consider the tree $\tau'=\tau\setminus x$ and the pair $(\tau', p(x))$ (where we still write $p(x)$ for the obvious image of the original vertex of $\tau$ in $\tau'$); it is clear that $d(\tau', p(x))=d(\tau, p(x))$ (because the only vertex removed is $x$, which was further from the root than $p(x)$) and so it equals $d(\tau,x)-1$.
\\
\end{tabularx}\\[12pt]
\begin{tabularx}{\textwidth}{p{.7cm}X}
{\tikz[baseline=15pt]
{
\path[use as bounding box] (-12pt,-15pt) rectangle (12pt,20pt);
\node[circle,inner sep=0pt,minimum size=4pt,fill=white,draw=black,label=left:$x$] (1) {};
\node[yshift=-15pt] (d) {$\ldots$};
\draw (1)-- (d);
\begin{pgfonlayer}{background}
\node[ellipse callout, draw, inner sep=2pt, fill=gray!10!white, callout absolute pointer={(1)}, above of=1, yshift=-10pt, callout pointer arc=110] {\phantom{$x$}};
\end{pgfonlayer}
}
}
 &  \textbf{[w\textsubscript{\textgreater0}]} If $x$ is a white vertex of $\tau$ and has offspring, the matter is more complicated. Proposition \ref{vlemma} ensures that a geodesic to the root moves to either $r(x)$ or $p(x)$, but it is not clear which: we need to distinguish two cases. 
\\
\end{tabularx}
\\[12pt]
\begin{tabularx}{\textwidth}{>{\raggedleft\arraybackslash}p{1.5cm}X}
\pgfdeclarelayer{background}
\pgfdeclarelayer{foreground}
\pgfsetlayers{background,main,foreground}
{\tikz[baseline=20pt]
{
\begin{pgfonlayer}{foreground}
\path[use as bounding box] (-12pt,-15pt) rectangle (12pt,28pt);
\node[circle,inner sep=0pt,minimum size=4pt,fill=white,draw=black,label=left:$x$,xshift=-5pt] (1) {};
\node[circle,inner sep=0pt,minimum size=4pt,yshift=-15pt,fill=gray,draw=black,label=left:$p(x)$,very thick] (2) {};
\node[circle,inner sep=0pt,minimum size=4pt,xshift=20pt,yshift=-10pt,fill=gray,draw=black, label=below:\scriptsize$r(p(x))$] (3) {};
\node[yshift=-25pt] (d) {$\ldots$};
\draw (1)-- (2)--(d);
\draw (2)--(3);
\end{pgfonlayer}
\node[ellipse callout, draw, inner sep=2pt, fill=gray!10!white, callout absolute pointer={(1)}, above of=1, yshift=-10pt, callout pointer arc=110] (tx) {\phantom{$x$}};
\node[ellipse callout, draw, inner sep=2pt, fill=gray!10!white, callout absolute pointer={(2)}, above right of=2, yshift=-10pt, xshift=-10pt, callout pointer arc=110] {\phantom{$x$}};
\node[ellipse callout, draw, inner sep=2pt, fill=gray!10!white, callout absolute pointer={(3)}, above of=3, yshift=-10pt, callout pointer arc=110] {\phantom{$x$}};
\begin{pgfonlayer}{background}
\node[ellipse, fit=(1)(tx), fill=red!90, draw=black, densely dashed, inner sep=0pt, fill opacity=0.1] (ellipse) {};
\node at (ellipse.north) {\raisebox{3pt}{\ding{33}}};
\end{pgfonlayer}
}
}
 & \textbf{[w\textsubscript{\textgreater0}.1]} Suppose that $x\neq r(p(x))$, that is $x$ has some right siblings; then $(p(x), r(p(x)))$ is a separating pair for $r(x)$, and thus $d(\tau, p(x))\leq d(\tau, r(x))$; hence there is a geodesic moving from $x$ to $p(x)$, and $d(\tau, p(x))=d(\tau, x)-1$. We choose to follow such a geodesic and define $\tau'$ to be $\tau \setminus \tau^x$, and output $(\tau', p(x))$ so that, as before, we have $d(\tau',p(x))=d(\tau,x)-1$.
\\
\end{tabularx}
\\[12pt]
\begin{tabularx}{\textwidth}{>{\raggedleft\arraybackslash}p{1.5cm}X}
\pgfdeclarelayer{background}
\pgfdeclarelayer{foreground}
\pgfsetlayers{background,main,foreground}
{\tikz[baseline=30pt]
{
\begin{pgfonlayer}{foreground}
\path[use as bounding box] (-5pt,-15pt) rectangle (25pt,40pt);
\node[circle,inner sep=0pt,minimum size=4pt,fill=white,draw=black,label=left:{$x$}] (1) {};
\node[circle,inner sep=0pt,minimum size=4pt,yshift=-15pt,fill=gray,draw=black,label=right:\scriptsize$p(x)$,very thick] (2) {};
\node[circle,inner sep=0pt,minimum size=4pt,yshift=10pt,xshift=15pt,fill=gray,draw=black,label=right:\scriptsize$r(x)$,very thick] (3) {};
\node[yshift=-25pt] (d) {$\ldots$};
\draw (3)--(1)-- (2)--(d);
\draw[red, <->] (3) .. controls (10pt,-10pt)  .. (2);
\end{pgfonlayer}
\node[ellipse callout, draw, inner sep=2pt, fill=gray!10!white, callout absolute pointer={(1)}, above of=1, yshift=-10pt, callout pointer arc=110] (tx) {\phantom{$x$}};
\node[ellipse callout, draw, inner sep=2pt, fill=gray!10!white, callout absolute pointer={(3)}, above of=3, yshift=-10pt, callout pointer arc=110] {\phantom{$x$}};
\begin{pgfonlayer}{background}
\node[ellipse, fit=(1)(tx), fill=red!90, draw=black, densely dashed, inner sep=0pt, fill opacity=0.1] (ellipse) {};
\node at (ellipse.north) {\raisebox{3pt}{\ding{33}}};
\end{pgfonlayer}
}
}
 & \textbf{[w\textsubscript{\textgreater0}.2]} Suppose now that $x=r(p(x))$. We build the new tree $\tau'$ by erasing $\tau^x$ from the original tree and rerooting the subtree $\tau^{r(x)}$ onto $p(x)$, by identifying $p(x)$ with $r(x)$, thus merging them into a single vertex $y$; the colour of $y$ is set to white if and only if both $p(x)$ and $t(x)$ were white in $\tau$. We need to show that $d(\tau', y)=d(\tau,x)-1$: there is an obvious map sending vertices of $\tau$ that are not in $\tau^x\setminus \tau^{r(x)}$ to vertices of $\tau'$; the map is 1-on-1 with the exceptions of $p(x)$ and $r(x)$, which are both sent to $y$. Neighbours in $\tau$ are sent to neighbours in $\tau'$, and (since $x=r(p(x))$ in $\tau$) the target of a vertex in $\tau$ becomes the target of its image in $\tau'$: hence $d(\tau',y)\leq \min\{d(\tau,p(x)),d(\tau, r(x))\}$. On the other hand, any map-path in $\tau'$ can be lifted to a map-path in $\tau$ (by appropriately selecting a pre-image for $y$ as first step of the path), which gives equality.
\\
\end{tabularx}
\\[12pt]
\begin{tabularx}{\textwidth}{>{\raggedleft\arraybackslash}p{.7cm}X}
{\tikz[baseline=15pt]
{
\path[use as bounding box] (-12pt,-15pt) rectangle (12pt,20pt);
\node[circle,inner sep=0pt,minimum size=4pt,fill=black,draw=black,label=left:$x$] (1) {};
\node[yshift=-15pt] (d) {$\ldots$};
\draw (1)-- (d);
\begin{pgfonlayer}{background}
\node[ellipse callout, draw, inner sep=2pt, fill=gray!10!white, callout absolute pointer={(1)}, above of=1, yshift=-10pt, callout pointer arc=110] {\phantom{$x$}};
\end{pgfonlayer}
}
}
 &  \textbf{[b]} If $x$ is a black vertex of $\tau$, then the new option of \emph{jumping} to $t(x)$ presents itself. As by Proposition \ref{vlemma}, a geodesic will either move to $p(x)$ or to $t(x)$. We need to deal with three separate cases:
\\
\end{tabularx}
\\[12pt]
\begin{tabularx}{\textwidth}{>{\raggedleft\arraybackslash}p{1.5cm}X}
\pgfdeclarelayer{background}
\pgfdeclarelayer{foreground}
\pgfsetlayers{background,main,foreground}
{\tikz[baseline=20pt]
{
\path[use as bounding box] (-12pt,-15pt) rectangle (12pt,28pt);
\begin{pgfonlayer}{foreground}
\node[circle,inner sep=0pt,minimum size=4pt,fill=black,draw=black,xshift=-15pt,label=left:$x$] (1) {};
\node[circle,inner sep=0pt,minimum size=4pt,yshift=-10pt,fill=gray,draw=black,label=left:\scriptsize$p(x)$,very thick] (2) {};
\node[circle,inner sep=0pt,minimum size=4pt,xshift=2pt,fill=gray,draw=black] (3) {};
\node[circle,inner sep=0pt,minimum size=4pt,xshift=10pt,yshift=-5pt,fill=gray,draw=black] (4) {};
\node[yshift=-20pt] (d) {$\ldots$};
\draw (1)-- (2)--(d);
\draw (4)--(2)--(3);
\end{pgfonlayer}
\node[ellipse callout, draw, inner sep=2pt, fill=gray!10!white, callout absolute pointer={(1)}, above of=1, yshift=-10pt, callout pointer arc=110] (tx) {\phantom{$x$}};
\node[ellipse callout, draw, inner sep=2pt, fill=gray!10!white, callout absolute pointer={(4)}, above of=4, yshift=-10pt, callout pointer arc=110] {\phantom{$x$}};
\node[ellipse callout, draw, inner sep=2pt, fill=gray!10!white, callout absolute pointer={(3)}, above of=3, yshift=-10pt, callout pointer arc=110] {\phantom{$x$}};
\begin{pgfonlayer}{background}
\node[ellipse, fit=(1)(tx), fill=red!90, draw=black, densely dashed, inner sep=0pt, fill opacity=0.1] (ellipse) {};
\node at (ellipse.north) {\raisebox{3pt}{\ding{33}}};
\end{pgfonlayer}
}
}
 & \textbf{[b.1]} Suppose $x$, $t(x)$ and $r(p(x))$ are distinct: that is to say, $x$ has at least two right siblings; in this case, $(p(x), r(p(x)))$ is a separating pair for $t(x)$, hence there is a geodesic moving from $x$ to $p(x)$; we thus set $\tau'$ to be $\tau \setminus \tau^x$, and output $(\tau', p(x))$.
\\
\end{tabularx}
\\[12pt]
\begin{tabularx}{\textwidth}{>{\raggedleft\arraybackslash}p{1.5cm}X}
\pgfdeclarelayer{background}
\pgfdeclarelayer{foreground}
\pgfsetlayers{background,main,foreground}
{\tikz[baseline=20pt]
{
\path[use as bounding box] (-12pt,-15pt) rectangle (12pt,28pt);
\begin{pgfonlayer}{foreground}
\node[circle,inner sep=0pt,minimum size=4pt,fill=black,draw=black,xshift=-10pt,label=left:$x$] (1) {};
\node[circle,inner sep=0pt,minimum size=4pt,yshift=-10pt,fill=gray,draw=black,label=left:$p(x)$,very thick] (2) {};
\node[circle,inner sep=0pt,minimum size=4pt,xshift=10pt,fill=gray,draw=black,very thick] (3) {};
\node[yshift=-20pt] (d) {$\ldots$};
\draw (1)-- (2)--(d);
\draw[red, <->] (2)--(3);
\end{pgfonlayer}
\node[ellipse callout, draw, inner sep=2pt, fill=gray!10!white, callout absolute pointer={(1)}, above of=1, yshift=-10pt, callout pointer arc=110] (tx) {\phantom{$x$}};
\node[ellipse callout, draw, inner sep=2pt, fill=gray!10!white, callout absolute pointer={(3)}, above of=3, yshift=-10pt, callout pointer arc=110] {\phantom{$x$}};
\begin{pgfonlayer}{background}
\node[ellipse, fit=(1)(tx), fill=red!90, draw=black, densely dashed, inner sep=0pt, fill opacity=0.1] (ellipse) {};
\node at (ellipse.north) {\raisebox{3pt}{\ding{33}}};
\end{pgfonlayer}
}
}
 & \textbf{[b.2]} Suppose now that $x$ has only one right sibling, which is therefore $t(x)$ as well as $r(p(x))$. We build the new tree $\tau'$ by simply erasing $\tau^x$ and identifying vertices $p(x)$ and $t(x)$, merging them into a single vertex $y$ to be coloured white if and only if both of the original vertices were white in $\tau$, and output $(\tau',y)$. We have $d(\tau',y)=d(\tau,x)-1$ by roughly the same argument as previously.
\\
\end{tabularx}
\\[12pt]
\begin{tabularx}{\textwidth}{>{\raggedleft\arraybackslash}p{1.5cm}X}
\pgfdeclarelayer{background}
\pgfdeclarelayer{foreground}
\pgfsetlayers{background,main,foreground}
{\tikz[baseline=10pt]
{
\path[use as bounding box] (-12pt,-15pt) rectangle (20pt,28pt);
\begin{pgfonlayer}{foreground}
\node[circle,inner sep=0pt,minimum size=4pt,fill=black,draw=black,label=left:$x$] (1) {};
\node[circle,inner sep=0pt,minimum size=4pt,yshift=-10pt,fill=gray,draw=black,label=left:$p(x)$] (2) {};
\node[yshift=-20pt] (d) {\scriptsize$\ldots$};
\node[circle,inner sep=0pt,minimum size=4pt,yshift=-30pt,fill=gray,draw=black,label=left:\scriptsize$p(t(x))$] (3) {};
\node[circle,inner sep=0pt,minimum size=4pt,yshift=-22pt,xshift=15pt,fill=gray,draw=black,label=below:\scriptsize$t(x)$,very thick] (4) {};
\draw (1)--(2)--(d)--(3)--(4);
\end{pgfonlayer}
\node[ellipse callout, draw, inner sep=2pt, fill=gray!10!white, callout absolute pointer={(1)}, above of=1, yshift=-10pt, callout pointer arc=110] (tx) {\phantom{$x$}};
\node[ellipse callout, draw, inner sep=2pt, fill=gray!10!white, callout absolute pointer={(4)}, above of=4, yshift=-18pt, xshift=8pt, callout pointer arc=110] (ttx) {\phantom{$x$}};
\node[yshift=5pt] at (tx.north) {\raisebox{3pt}{\ding{33}}};
\begin{pgfonlayer}{background}
\draw[dashed, right color=red!20, left color=red!2] (tx.north)+(-10pt,5pt) .. controls (35pt,30pt) and (20pt,-30pt) .. (-10pt,-25pt);
\end{pgfonlayer}
}
}
 & \textbf{[b.3]} The last case is that of $x$ being the rightmost child of its parent; in this case $(p(t(x)),t(x))$ is a separating pair for $p(x)$, so $d(\tau, t(x))\leq d(\tau, p(x))$, and we may assume a geodesic to the root does jump from $x$ to its target. We can thus build $\tau'$ by erasing \emph{all that lies left of $t(x)$}, and output $(\tau',t(x))$.\\
\end{tabularx}\\[12pt]

Notice that, in all cases listed except for the very last one, the output vertex $x'$ has height $|x|-1$ in $\tau'$, whereas the jump made in the last case (the one marked [b.3]) may lead to a vertex $x'$ of arbitrarily smaller height. Also, the information on $\tau$ that the algorithm uses to select the appropriate $\tau'$ is entirely local (child structure of $x$ and its parent) with the exception of the last case, which requires to make changes to parts of $\tau$ that are, a priori, arbitrarily far from $x$.

In the spirit of making each step by the algorithm entirely determined by local information, which will in turn entail precious independence properties as soon as we switch to a random setting, we add extra data to inputs and outputs: we let the algorithm run on triples of the form $(\tau, x, s)$, where $s$ is one of four \emph{states}, and output a triple $(\tau',x',s')$.

Three of the states simply mimic the cases listed above: we call them $w_0$, $w_{>0}$ and $b$; a fourth state, labelled $j$ for \emph{jump}, is devised to deal specifically with situations that fall under case [b.3]: the idea is that, instead of simply outputting an entirely different tree paired with the target of the jump, the algorithm goes into a \emph{jump} state; it proceeds modifying the tree a little at a time until it reaches the original target, at which point it `lands' in one of the three non-jump states.

We propose to re-define outputs according to the state of the input triple, keeping in mind that they mostly adhere to the preceding description for pairs; if the output state in the triple $(\tau',x',s')$ is known \emph{not} to be $j$, then it is determined by $(\tau',x')$. The initial state, in particular, is not jump: given $\tau$ and a vertex $x$, it can be determined as being
\begin{itemize}
\item[$\circ$] $w_0$, if $x$ is a white leaf of $\tau$;
\item[$\circ$] $w_{>0}$, if $x$ is white and it is not a leaf in $\tau$;
\item $b$, if $x$ is black.
\end{itemize}

The behaviour of the algorithm for an input triple of the form $(\tau, x, w_0)$ or $(\tau, x, w_{>0})$, and for input triples $(\tau, x, b)$ with $x$ having one or more right siblings, is exactly that described in cases [w\textsubscript{0}], [w\textsubscript{\textgreater0}.1], [w\textsubscript{\textgreater0}.2], [b.2] and [b.1]. Namely, the tree $\tau'$ and a vertex $x'$ are produced, and the state $s'$ is selected again among the three non-jump states ($w_0$, $w_{>0}$, $b$) according to the colour and degree of the output vertex $x'$ (white leaf, white non-leaf, black). 

\pgfdeclarelayer{background}
\pgfdeclarelayer{foreground}
\pgfsetlayers{background,main,foreground}

We now describe the behaviour of the algorithm when the input is of the form $(\tau,x,b)$, and $x$ is the rightmost child of its parent. As explained earlier (case [b.3]), the next vertex in a map-geodesic to the root would be (without loss of generality) $t(x)$. We distinguish yet two subcases.
\\[12pt]
\begin{tabularx}{\textwidth}{>{\raggedleft\arraybackslash}p{1.5cm}X}
{\tikz[baseline=10pt]
{
\path[use as bounding box] (-12pt,-15pt) rectangle (18pt,28pt);
\begin{pgfonlayer}{foreground}
\node[circle,inner sep=0pt,minimum size=4pt,fill=black,draw=black,label=left:$x$] (1) {};
\node[circle,inner sep=0pt,minimum size=4pt,yshift=-10pt,fill=gray,draw=black,label=left:$p(x)$] (2) {};
\node[circle,inner sep=0pt,minimum size=4pt,yshift=-20pt,fill=gray,draw=black] (3) {};
\node[circle,inner sep=0pt,minimum size=4pt,yshift=-12pt,xshift=13pt,fill=gray,draw=black,very thick] (4) {};
\node[yshift=-28pt] (d) {$\ldots$};
\draw (1)-- (2)--(3)--(d);
\draw (3)--(4);
\end{pgfonlayer}
\node[ellipse callout, draw, inner sep=2pt, fill=gray!10!white, callout absolute pointer={(1)}, above of=1, yshift=-10pt, callout pointer arc=110] (tx) {\phantom{$x$}};
\node[ellipse callout, draw, inner sep=2pt, fill=gray!10!white, callout absolute pointer={(4)}, above of=4, yshift=-10pt, callout pointer arc=110] {\phantom{$x$}};
\begin{pgfonlayer}{background}
\node[ellipse, fit=(1)(tx), fill=red!90, draw=black, densely dashed, inner sep=2pt, fill opacity=0.1,yshift=-4pt] (ellipse) {};
\node at (ellipse.north) {\raisebox{3pt}{\ding{33}}};
\end{pgfonlayer}
}
}
 & \textbf{[b.3.1]} If $t(x)$ has height $|x|-1$ (that is, if $p(x)$ has a right sibling: again, a local property) then define $\tau'$ by erasing $\tau^{p(x)}$ from $\tau$ and output $(\tau', t(x), s')$, where $s'$ is -- again -- one of $w_0$, $w_{>0}$, $b$, according to properties of $t(x)$ in $\tau'$.\\
\end{tabularx}\\[12pt]
\begin{tabularx}{\textwidth}{>{\raggedleft\arraybackslash}p{1.5cm}X}
{\tikz[baseline=20pt]
{
\begin{pgfonlayer}{foreground}
\path[use as bounding box] (-12pt,-15pt) rectangle (12pt,28pt);
\node[circle,inner sep=0pt,minimum size=4pt,fill=black,draw=black,label=left:$x$] (1) {};
\node[circle,inner sep=0pt,minimum size=4pt,yshift=-12pt,fill=gray,draw=black,label=left:$p(x)$,very thick] (2) {};
\node[circle,inner sep=0pt,minimum size=4pt,yshift=-22pt,fill=gray,draw=black] (3) {};
\node[yshift=-30pt] (d) {$\ldots$};
\draw (1)-- (2)--(3)--(d);
\end{pgfonlayer}
\node[ellipse callout, draw, inner sep=2pt, fill=gray!10!white, callout absolute pointer={(1)}, above of=1, yshift=-10pt, callout pointer arc=110](tx) {\phantom{$x$}};
\begin{pgfonlayer}{background}
\node[ellipse, fit=(1)(tx), fill=red!90, draw=black, densely dashed, inner sep=0pt, fill opacity=0.1] (ellipse) {};
\node at (ellipse.north) {\raisebox{3pt}{\ding{33}}};
\end{pgfonlayer}
}
}
 & \textbf{[b.3.2]} If, however, $p(x)$ has no right siblings, then it is time to finally put the jump state $j$ to use. We define $\tau'$ to be $\tau \setminus \tau^x$ and output $(\tau', p(x), j)$. Notice that now $p(x)$ has height $|x|-1$ in $\tau'$, and that the identity of $t(x)$ can still be recovered in $\tau'$ (even though vertex $x$ has been erased) by the sole knowledge of $p(x)$: if $p(x)$ were black, $t(x)$ would be the target of $p(x)$.
\\
\end{tabularx}\\[12pt]
We finally give instructions for the algorithm to follow when confronted with a jump state. Suppose we have an input $(\tau,x,j)$; \\
\begin{tabularx}{\textwidth}{>{\raggedleft\arraybackslash}p{1.5cm}X}
{\tikz[baseline=0pt]
{
\path[use as bounding box] (-12pt,-15pt) rectangle (12pt,28pt);
\node[circle,inner sep=0pt,minimum size=4pt,fill=gray,draw=black,label=left:$x$] (1) {};
\node[circle,inner sep=0pt,minimum size=4pt,yshift=-10pt,fill=gray,draw=black,label=left:$p(x)$,very thick] (2) {};
\node[circle,inner sep=0pt,minimum size=4pt,yshift=-20pt,fill=gray,draw=black] (3) {};
\node[yshift=-28pt] (d) {$\ldots$};
\draw (1)-- (2)--(3)--(d);
\begin{pgfonlayer}{background}
\node[ellipse, fit=(1), fill=red!90, draw=black, densely dashed, inner sep=2pt, fill opacity=0.1] (ellipse) {};
\node at (ellipse.north) {\raisebox{3pt}{\ding{33}}};
\end{pgfonlayer}
}
}
& \textbf{[j.1]} if $p(x)$ has no right siblings, then output $(\tau\setminus\tau^{x},p(x),j)$; this way, the vertex on which the geodesic should land is still the target of $p(x)$ (or would be if $p(x)$ were black) and its map-distance from the root remains unchanged;\\
\end{tabularx}\\
\begin{tabularx}{\textwidth}{>{\raggedleft\arraybackslash}p{1.5cm}X}
{\tikz[baseline=0pt]
{
\path[use as bounding box] (-12pt,-15pt) rectangle (16pt,28pt);
\node[circle,inner sep=0pt,minimum size=4pt,fill=gray,draw=black,label=left:$x$] (1) {};
\node[circle,inner sep=0pt,minimum size=4pt,yshift=-10pt,fill=gray,draw=black,label=left:$p(x)$] (2) {};
\node[circle,inner sep=0pt,minimum size=4pt,yshift=-20pt,fill=gray,draw=black] (3) {};
\node[circle,inner sep=0pt,minimum size=4pt,yshift=-12pt,xshift=12pt,fill=gray,draw=black,very thick] (4) {};
\node[yshift=-28pt] (d) {$\ldots$};
\draw (1)-- (2)--(3)--(d);
\draw (3)--(4);
\begin{pgfonlayer}{background}
\node[ellipse, fit=(1)(2), fill=red!90, draw=black, densely dashed, inner sep=2pt, fill opacity=0.1,yshift=2pt] (ellipse) {};
\node at (ellipse.north) {\raisebox{3pt}{\ding{33}}};
\node[ellipse callout, draw, inner sep=2pt, fill=gray!10!white, callout absolute pointer={(4)}, above of=4, yshift=-10pt, callout pointer arc=110] {\phantom{$x$}};
\end{pgfonlayer}
}
}
 & \textbf{[j.2]} if $p(x)$ has right siblings, then the leftmost one (call it $x'$) is the (image of the) vertex the geodesic was supposed to land on; output $(\tau\setminus\tau^{p(x)},x',s')$, with $s'$ being appropriately chosen among the three non-jump states.\\
\end{tabularx}\\[12pt]

We now summarise the key properties of the algorithm as just described: given a well bicoloured tree $\tau$ and a vertex $x$ of height $n$, we can generate a sequence of $n+1$ triples $(\tau_i,x_i,s_i)$, with $i$ ranging from $0$ to $n$, where
\begin{itemize}
\item for each $i$, $\tau_i$ is a well bicoloured tree, $x_i$ is a vertex of $\tau$ and $s_i$ is one of four states ($w_0$, $w_{>0}$, $b$, $j$);
\item $\tau_0$ is the tree $\tau$ deprived of all that lies left of the (tree) path from the root to $x$, and $x_0=x$;
\item $(\tau_{i+1},x_{i+1},s_{i+1})$ is obtained from $(\tau_i,x_i,s_i)$ as described, with changes of a `local' nature;
\item for each $i$, the height of $x_i$ in $\tau_i$ is $n-i$; consequently, $x_n$ is the root of $\tau_n$;
\item for each $i$ between 0 and $n-1$ such that $s_i\neq j$, $d(\tau_{i+1},x_{i+1})=d(\tau_i,x_i)-1$; therefore, 
\begin{equation}\label{nminusjump}d(\tau,x)=\sum_{0\leq i<n|s_i\neq j}1=n-|\{0\leq i<n|s_i= j\}|;\tag{$\star$}\end{equation}
that is, the distance $d(\tau,x)$ is the number of non-jump states appearing in the input triples on which the algorithm runs (indeed, it is also the number of non-jump states appearing in the $n$ \emph{output} triples, because $s_n$ cannot be $j$ and neither can $s_0$). We will make frequent use of this fact in what follows.
\end{itemize}

The time has come to run our algorithm on a random tree; in order to do this, an especially useful tool is the standard construction of the \emph{geometric Galton-Watson tree conditioned to survive}, which will provide us with a way of unifying results given by the algorithm for vertices of arbitrary height into a single random variable. The next section is an introduction to this tool and to some of the notation needed for further progress.

\section{The Galton-Watson tree conditioned to survive}\label{infinite tree}

In this section we briefly introduce the critical Galton-Watson tree conditioned to survive for a geometric offspring distribution; for a more general definition and further details see Section~12 of~\cite{lyons}.

We build a random infinite tree, called $T_\infty$, in the following way: consider an infinite path $v_0v_1\ldots v_n\ldots$, called the \emph{spine}; let $(\lambda^i)_{i\geq 0}$ be a sequence of independent critical geometric Galton-Watson trees (that is, with offspring distribution $\mu$, where $\mu(k)=2^{-k-1}$ for all natural numbers $k$), and let $(\rho^i)_{i\geq 0}$ be another such sequence, independent of the first; we consider the Galton-Watson trees as random rooted plane trees and, for each $i\geq 0$, attach $\lambda^i$ to the spine by identifying its root and $v_i$, so that $\lambda^i$ lies to the left of the spine; similarly, attach each $\rho^i$ so that its root is identified with $v_i$ and all of its vertices lie to the right of the spine (see Figure \ref{right and left}); finally, root the random infinite tree thus obtained in $v_0$. 

The relevance of critical geometric Galton-Watson trees lies in the fact that, if $\theta$ is one such tree and $\theta_n$ is the random tree obtained by conditioning $\theta$ on having exactly $n$ vertices, then $\theta_n$ is uniformly distributed over plane trees with $n$ vertices; as a result, the tree $T_\infty$ itself has much to do with random plane trees, as we shall now see.

\begin{figure}
\centering
\begin{subfigure}[b]{.3\textwidth}
\centering
\begin{tikzpicture}[node distance=25 pt, b/.style={circle,inner sep=0pt,minimum size=5pt,fill=black,draw=black}, w/.style={circle,inner sep=0pt,minimum size=5pt, draw=black, fill=white}, g/.style={circle,inner sep=0pt,minimum size=5pt, draw=black, fill=gray}]
\node [b, label=below :$v_3$] (3) {};
\node [above of=3] (dots) {$\ldots$};
\node [b, below of=3, label=below:$v_2$] (2) {};
\node [b, below of=2, label=below:$v_1$] (1) {};
\node [b, below of=1, label=below:$v_0$] (0) {};
\draw (dots)--(3)--(2)--(1)--(0);
\begin{pgfonlayer}{background}
\foreach \x in {0,...,3} {
\node[ellipse callout, draw, inner sep=1pt, fill=gray!10!white, callout absolute pointer={(\x)}, right of=\x, callout pointer arc=130] {$\rho^\x$};
\node[ellipse callout, draw, inner sep=2pt, fill=gray!10!white, callout absolute pointer={(\x)}, left of=\x, callout pointer arc=110] {$\lambda^\x$};}
\end{pgfonlayer}
\end{tikzpicture}
\caption{\label{right and left}}
\end{subfigure}
\begin{subfigure}[b]{.3\textwidth}
\centering
\begin{tikzpicture}[node distance=25 pt, b/.style={circle,inner sep=0pt,minimum size=5pt,fill=black,draw=black}, w/.style={circle,inner sep=0pt,minimum size=5pt, draw=black, fill=white}, g/.style={circle,inner sep=0pt,minimum size=5pt, draw=black, fill=gray}]
\node [b] (2) {};
\node [b, below of=2] (1) {};
\node [b, below of=1] (0) {};
\draw (2)--(1)--(0);
\begin{pgfonlayer}{background}
\foreach \x in {0,1} {
\node[ellipse callout, draw, inner sep=1pt, fill=gray!10!white, callout absolute pointer={(\x)}, right of=\x, callout pointer arc=130] {$\rho^\x$};
\node[ellipse callout, draw, inner sep=2pt, fill=gray!10!white, callout absolute pointer={(\x)}, left of=\x, callout pointer arc=110] {$\lambda^\x$};}
\node[ellipse callout, draw, inner sep=1pt, fill=gray!10!white, callout absolute pointer={(2)}, right of=2, callout pointer arc=130] {$\rho^2$};
\draw node [circle, fill=red, fill opacity=.5, minimum size=2pt] at (2) {}; 
\end{pgfonlayer}
\end{tikzpicture}
\caption{}
\end{subfigure}
\caption{Figure (a) represents the critical geometric Galton-Watson tree conditioned to survive; trees $\rho^1,\ldots, \rho^n, \ldots$ and $\lambda^1,\ldots,\lambda^n,\ldots$ are independent geometric Galton-Watson trees. Figure (b) represents $cut_2(T_\infty)$.}\label{Tinfty}
\end{figure}
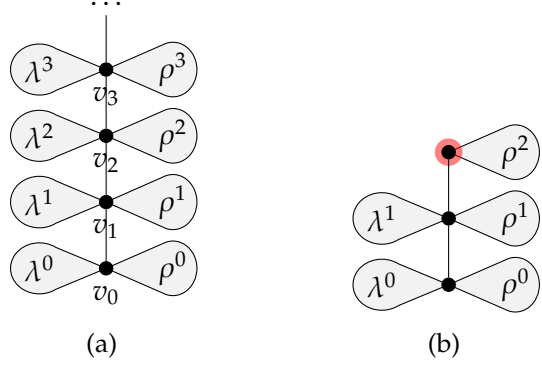

We write $cut_n(T_\infty)$ for the pair $(\tau,v_n)$, where $\tau$ is the (finite plane rooted) tree obtained from $T_\infty$ by erasing every descendant of $v_n$ which lies \emph{strictly to the left} of $v_{n+1}$, together with $v_{n+1}$ itself and all of its descendants: that is, $\tau$ consists of the path $v_0\ldots v_n$, with trees $\rho^0,\ldots ,\rho^n$ attached to the right and trees $\lambda^0,\ldots,\lambda^{n-1}$ attached to the left. Then we have the following standard result.

\begin{lem}\label{go infinite}Let $F$ be a non-negative real valued function defined on pairs $(\tau,u)$, where $\tau$ is a finite rooted plane tree and $u$ is a vertex in $\tau$; let $\theta$ be a Galton-Watson tree with critical geometric offspring distribution. Then for all $n\geq 0$
$$\GW\left[F\left(cut_n(T_{\infty})\right)\right]=\GW\left[\sum_{\substack{u\in \theta\\|u|=n}}F(\theta,u)\right].$$
\end{lem}

\begin{proof}
For the left hand side we have
$$\GW\left[F\left(cut_n(T_{\infty})\right)\right]=\sum_{(\tau,u)}F(\tau, u)\GW(cut_n(T_\infty)=(\tau,u))=$$
$$\displaystyle{=\sum_{\tau}\sum_{\substack{u\in \tau\\|u|=n}}F(\tau,u)\GW(cut_n(T_\infty)=(\tau,u))}$$
where $(\tau, u)$ ranges among all pairs formed by a finite (rooted plane) tree and a vertex $u$ of height $n$ in the tree.

Thus it is enough to show that, for all such pairs $(\tau,u)$,
$$\GW(cut_n(T_\infty)=(\tau,u))=\GW(\theta=\tau).$$

But the probability $\GW(cut_n(T_\infty)=(\tau,u))$ is easy to compute: consider the set formed by $u$ and its ancestors in $\tau$; order them according to height, and label them $u_0,\ldots, u_n$, so that $|u_i|=i$ ($u_0$ is the root and $u_n$ is $u$). For $i=0,\ldots,n-1$, let $\tau_l^i$ be the subtree of $\tau$ formed by $u_i$ and its descendants lying strictly to the left of $u_{i+1}$. Similarly, let $\tau_r^i$ be the subtree of $\tau$ formed by $u_i$ and its descendants lying strictly to the right of $u_{i+1}$. Then

$$\GW(cut_n(T_\infty)=(\tau,u))=\left(\prod_{i=0}^{n-1}\GW(\lambda^i=\tau_l^i)\GW(\rho^i=\tau_r^i)\right) \cdot \GW(\rho^n=\tau^{u})=$$
$$=\left(\prod_{i=0}^{n-1}\GW(\theta=\tau_l^i)\GW(\theta=\tau_r^i)\right)\cdot \GW(\theta=\tau^{u}).$$

Now, for each vertex $x$ in $\tau$, let $c(x)$ be the number of children of $x$ in $\tau$; clearly,
$$\GW(\theta=\tau)=\prod_{x\in\tau}2^{-c(x)-1}$$
by definition of $\theta$.

On the other hand, for each vertex in tree $\tau_l^i$ ($i$ between 1 and $n-1$) call $c_l(x)$ the number of its children in $\tau_l^i$; similarly, call $c_r(x)$ the number of children of $x$ in $\tau_r^i$, if $x$ belongs to such a tree. Then
$$\GW(cut_n(T_\infty)=(\tau,u))=\prod_{i=0}^{n-1}\left\{\prod_{x\in\tau_l^i}2^{-c_l(x)-1}\prod_{x\in\tau_r^i}2^{-c_r(x)-1}\right\}\cdot \prod_{x\in\tau^u}2^{-c(x)-1}.$$

Consider any vertex $x$ of $\tau$ such that $x\not\in \{u_0, \ldots u_{n-1}\}$; then $x$ appears only once in the expression above, as part of some tree $\tau_l^j$ or $\tau_r^j$, or possibly of $\tau^u$. Furthermore, the number of children of $x$ in its subtree ($c_l(x)$ or $c_r(x)$ or $c(x)$) is exactly the same as the number of children $c(x)$ that $x$ has in $\tau$. Now, for $i=0,\ldots,n-1$, consider $u_i$; it appears inside two of the products, as part of tree $\tau_l^i$ and tree $\tau_r^i$, therefore it contributes to $\GW(cut_n(T_\infty)=(\tau,u))$ with a factor $2^{-c_l(u_i)-1}2^{-c_r(u_i)-1}$; on the other hand, $u_i$ has $c(u_i)=c_l(u_i)+c_r(u_i)+1$ children in $\tau$, so that its contribution to $\GW(\theta=\tau)$ is a factor $2^{-c(u_i)-1}=2^{-c_l(u_i)-c_r(u_i)-2}$.

Consequently, we have $\GW(cut_n(T_\infty)=(\tau,u))=\GW(\theta=\tau)$, as wanted.
\end{proof}

\bigskip

In order to adapt the notion of the critical geometric Galton-Watson tree conditioned to survive to our prior setting, we need to endow it with a random bicolouring. We do this by simply choosing the colour for each vertex of $T_\infty$ uniformly at random with probability $1/2$. In what follows, since the context will determine whether or not trees are bicoloured, we will still write $T_\infty$ for the object just introduced, namely the \emph{critical geometric Galton-Watson tree conditioned to survive, uniformly bicoloured}.

Similarly, if $\tau$ is a (random) finite tree, it can be uniformly bicoloured by choosing a colour for each of its vertices independently and uniformly at random, conditionally on $\tau$ itself. This applies in particular when $\tau$ is a critical geometric Galton-Watson tree. 

Finally, we remark that a bicoloured critical geometric Galton-Watson tree is not necessarily \emph{well} bicoloured; given a finite bicoloured tree $\tau$, we define the new tree $\tau^\circ$ as the one obtained by adding a white leaf as rightmost child of the root of $\tau$ and recolouring the root white, so that $\tau^\circ$ is well bicoloured; we will always think of $\tau$ as being embedded in $\tau^\circ$ in the obvious way. One may of course consider map-distances on $\tau^\circ$: given a vertex $u$ in $\tau$, we write $d^\circ(\tau,u)=d(\tau^\circ, u)$ for the map distance between vertex $u$ and the root in $\tau^\circ$. Analogously, we may consider $d^\circ(cut_n(T_\infty))$, which, if $cut_n(T_\infty)=(\tau,u)$, we take to mean $d(\tau^\circ,u)$.

\section{The algorithm running on the infinite bicoloured tree}\label{markov}

Let $T_\infty$ be the critical geometric Galton-Watson tree conditioned to survive, uniformly bicoloured; call $v_0,\ldots,v_n,\ldots$ the vertices on its spine ($v_0$ being the root). The distance $d^\circ(cut_n(T_\infty))$ is, for each positive integer $n$, a random variable which we wish to estimate (at least asymptotically in $n$).

We have described all through Section \ref{algorithm} an algorithm that can now be started on $(\tau^\circ,x,s)$, where $(\tau,x)=cut_n(T_\infty)$ and $s$ depends on the colour and offspring of $x$ in $\tau$ (therefore on $T_\infty$ and $n$); this will yield a sequence of $n$ states $s_0, \ldots , s_{n-1}$, each a random variable taking values in the space of states $\Sigma=\{w_0,w_{>0},b,j\}$, such that
$$d^\circ(cut_n(T_\infty))=n-|\{0\leq i <n|s_i=j\}|.$$

This sequence $(s_i)_{0\leq i < n}$ is `almost' a Markov Chain, in the sense clarified by the following fundamental proposition:

\begin{prop}\label{markovchain} Fix $n>0$; take the random infinite tree $T_\infty$ and consider the sequence $s_0,\ldots,s_{n-2}$ of the first $n-1$ inputs for the algorithm from Section \ref{algorithm}, started on $(\tau^\circ,x,s_0)$, where $(\tau,x)=cut_n(T_\infty)$. Then such a sequence has the same law as (the first $n-1$ steps of) a Markov chain $(X_i)_{i\geq 0}$ with transition matrix
$$
M=
\left(\begin{array}{c|cccc}
\Rsh & w_0 & w_{>0} & b & j\\
\hline
w_0 & 1/4 & 1/4 & 1/2 & 0 \\
w_{>0} & 1/16 & 5/16 & 5/8 & 0\\
b & 3/32 & 7/32 & 7/16 & 1/4\\
j & 1/8 & 1/8 & 1/4 & 1/2\\
\end{array}\right)$$
and a random initial state distributed as $(1/4, 1/4, 1/2, 0)$.
\end{prop}

This proposition plays a key role in finally establishing Theorem \ref{main}; it is, in fact, the motivation that led to the algorithm as described in Section \ref{algorithm}, and its proof is nothing but a careful observation of how the various steps of the algorithm interact with the random element: in particular, some key independence properties are always preserved, mainly thanks to the Galton-Watson structure of subtrees and the nature of the geometric law.

\begin{wrapfigure}{r}{0.3\textwidth}
\vspace{-10pt}
\centering
\begin{tikzpicture}[node distance=25 pt, b/.style={circle,inner sep=0pt,minimum size=5pt,fill=black,draw=black}, w/.style={circle,inner sep=0pt,minimum size=5pt, draw=black, fill=white}, g/.style={circle,inner sep=0pt,minimum size=5pt, draw=black, fill=gray}]
\node [g, label=left :$p(x_i)$] (3) {};
\node [below of=3] (dots) {$\ldots$};
\node [g, above of=3, label=above :$x_i$] (i) {};
\node [g, below of=dots, label=left:$p^{(n-i-2)}(x_i)$] (2) {};
\node [g, below of=2, label=left:$p^{(n-i-1)}(x_i)$] (1) {};
\node [g, below of=1, label=left:$p^{(n-i)}(x_i)$] (0) {};
\draw (i)--(3)--(dots)--(2)--(1)--(0);
\begin{pgfonlayer}{background}
\foreach \x in {0,...,2} {
\node[ellipse callout, draw, inner sep=1pt, fill=gray!10!white, callout absolute pointer={(\x)}, right of=\x, callout pointer arc=130] {$\rho_i^\x$};}
\node[ellipse callout, draw, inner sep=1pt, fill=gray!10!white, callout absolute pointer={(i)}, right of=i, callout pointer arc=110] {$\rho_i^{n-i}$};
\node[ellipse callout, draw, inner sep=1pt, fill=gray!10!white, callout absolute pointer={(3)}, right of=3,xshift=5pt, callout pointer arc=110] {$\rho_i^{n-i-1}$};
\end{pgfonlayer}
\end{tikzpicture}
\caption{\label{right trees}Structure of $\tau_i$.}
\end{wrapfigure}
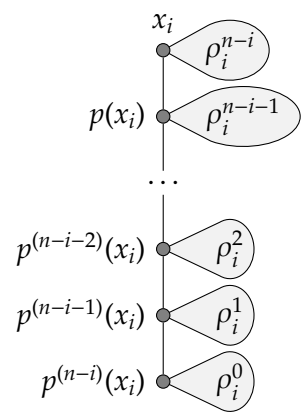

In order to prove Proposition \ref{markovchain}, we highlight those exact independence properties in a separate Lemma, for which some additional notation is needed. Suppose the algorithm is started on a triple $(\tau_0,x_0,s_0)$, where $x_0$ has height $n$ in $\tau_0$; we write $p^{j}$ for the $j$-th iteration of the parent function, so that $p^0(x_0)=x_0$ and $p^j(x_0)$ has height $n-j$ (in particular, $p^n(x_0)$ is the root of $\tau_0$). We call $\rho_0^j$ the subtree of $\tau_0$ consisting of ${p^{n-j}(x_0)}$ and its descendants lying \emph{strictly to the right} of ${p^{n-j-1}(x_0)}$ (so that the tree $\rho_0^j$ is rooted in a vertex of height $j$ in $\tau_0$); $\rho_0^n$ is simply $\tau_0^{x_0}$.

We repeat the same construction for all subsequent triples $(\tau_i,x_i,s_i)$: $x_i$ has height $n-i$; we write $\rho_i^j$ for the subtree of $\tau_i$ consisting of ${p^{n-i-j}(x_i)}$ and its descendants lying \emph{strictly to the right} of ${p^{n-i-j-1}(x_i)}$ (see Figure \ref{right trees}); as before, the root of $\rho_i^j$ has height $j$ in $\tau_i$, and $\rho_i^{n-i}=\tau_i^{x_i}$. 

We wish to prove the following:

\begin{lem}\label{independence}Fix $n>0$ and consider the (random) triples $(\tau_i,x_i,s_i)$ obtained from $T_\infty$ through the algorithm for $i$ between 0 and $n-2$. Then for each $i$,
\begin{itemize}
\item if $s_i=w_0$ or $s_i=b$, then $\rho_i^j$, for $j=1,\ldots ,n-i-1$, is a sequence of independent uniformly bicoloured Galton-Watson trees;
\item if $s_i=w_{>0}$, then $\rho_i^j$, for $j=1,\ldots ,n-i-1$, is a sequence of independent uniformly bicoloured Galton-Watson trees; also, $\tau_i^{r(x_i)}$ is a uniformly bicoloured Galton-Watson tree independent of the block of the $\rho_i^j$'s;
\item if $s_i=j$, then $\rho_i^j$, for $j=1,\ldots ,n-i-2$, is a sequence of independent uniformly bicoloured Galton-Watson trees, whilst $\rho_i^{n-i-1}=\{p(x_i)\}$ and $\rho_i^{n-i}=\{x_i\}$.
\end{itemize} 
\end{lem}

\begin{proof}
We proceed by induction on $i$; for $i=0$, all assertions are trivial by definition of $T_\infty$ (more precisely, by the random structure of $cut_n(T_\infty)$): with the exception of $\rho^0_0$, whose root is recoloured as white and has a white leaf attached as a rightmost child, $\rho^j_0=\rho^j$: all right trees attached to the spine of $T_\infty$ are uniformly bicoloured Galton-Watson trees.

We loosely follow the original presentation of the algorithm and deal separately with each case.

Suppose $s_i=w_0$, and consider $(\tau_{i+1},x_{i+1},s_{i+1})$; we have $\rho_{i+1}^j=\rho_{i}^j$ for $j=1,\ldots,n-i-1$, since all the algorithm does is erase $x_i$ (that is, $\rho_{i}^{n-i}$), and the claim follows by the induction hypothesis.

If $s_i=w_{>0}$, then the output of the algorithm depends on whether $\rho_i^{n-i-1}$ consists of only $p(x_i)$ or not. If $s_{i+1}=b$ or $s_{i+1}=w_0$, then $\rho_{i+1}^j=\rho_{i}^j$ for $j=1,\ldots,n-i-2$, which is all that is required. The same is true for $s_{i+1}=w_{>0}$, but we also need to show that $\tau_{i+1}^{r(x_{i+1})}$ is Galton-Watson and independent of the $\rho_{i+1}^j$'s. The subtree $\rho_{i+1}^{n-i-1}$ is determined by the algorithm as follows. If $\rho_i^{n-i-1}=\{p(x_i)\}$ then it is isomorphic to $\tau_i^{r(x_i)}$ (we are interested in only the case of $r(x_i)$ and $p(x_i)$ being white), which has the required properties by the induction hypothesis. Otherwise we simply have $\rho_{i+1}^{n-i-1}=\rho_i^{n-i-1}$, and we are done.

If $s_i=b$, then we need to deal with a few cases separately. The most straightforward one is that of $s_{i+1}=j$; what the algorithm does in this case is merely erase $\tau_i^{x_i}$: conditions on trees $\rho_{i+1}^j$ are automatic (including the fact that $\rho_{i+1}^{n-i-2}=\{p(x_{i+1})\}$ and $\rho_{i+1}^{n-i-1}=\{x_{i+1}\}$, or we would not have switched to jump state). If $s_{i+1}=w_0$, this may be for one of two reasons:
\begin{itemize}
\item $x_i$ has only one right sibling in $\tau_i$, which is a white leaf, and $p(x_i)$ is white (case [b.2]); there is nothing to prove here, since $\rho_{i+1}^j=\rho_{i}^j$ for $j=1,\ldots,n-i-2$;
\item $x_i$ has no right sibling in $\tau_i$, but $p(x_i)$ does, and its next sibling is a white leaf (case [b.3.1]); in this case, while trees $\rho^{j}_{i}$ remain unchanged for $j\leq n-i-3$, $\rho_{i+1}^{n-i-2}$ is $\rho_{i}^{n-i-2}$ with its leftmost branch erased; on the other hand, what we have done is precisely condition such a tree on having a leftmost branch made up of a white leaf, and then remove it, which leaves nothing but a Galton-Watson tree.
\end{itemize}

The case of $s_{i+1}=b$ presents no added difficulties apart for the need for more casework. One needs to deal separately with cases [b.2], [b.3.1] (same as above, with a weaker condition to check) and [b.1], 	which is again trivial. Finally, one has to go through essentially the same for $s_{i+1}=w_{>0}$, but with the added requirement to show that $\tau_{i+1}^{r(x_{i+1})}$ is Galton-Watson and independent of the $\rho_{i+1}^j$'s. This is true in case [b.1] ($\tau_{i+1}^{r(x_{i+1})}=\tau_i^{r(p(x_i))}$), [b.2] ($\tau_{i+1}^{r(x_{i+1})}=\tau_i^{r(r(p(x_i)))}$, where $r(r(p(x_i)))$ does exist in $\tau_i$, or we would not go to state $w_{>0}$), [b.3.1] ($\tau_{i+1}^{r(x_{i+1})}=\tau_i^{r(r(p(p(x_i))))}$).

The very last possibility is for $s_i$ to be $j$; if $s_{i+1}$ is $j$ as well, there is hardly anything to prove; all other cases require arguments that are exactly the same as those used for $s_i=b$, and that we shall not repeat.

This concludes the proof by induction.
\end{proof}

Given Lemma \ref{independence}, Proposition \ref{markovchain} is only a matter of computing transition probabilities. We refer the reader to the summary table in the next page.

\afterpage{\include{tables1}}

\bigskip

The purpose of the algorithm was, since the very beginning, to give estimates for the map-distance of vertices from the root; we are now in a position to easily obtain asymptotics for the distance $d^\circ(cut_n(T_\infty))$. Namely, we have

\begin{prop}\label{markovchainestimate}Let $T_\infty$ be the geometric Galton-Watson tree conditioned to survive, uniformly bicoloured; then
$$\lim_{n\rightarrow\infty} \frac{d^\circ(cut_n(T_\infty))}{n}=7/9,$$
where the convergence is almost sure. Moreover, for all $\epsilon>0$ there exist positive $n_{\epsilon}, C$ such that, for all $n\geq n_{\epsilon}$,
$$\mP\left(\left|\frac{d^\circ(cut_n(T_\infty))}{n}-\frac{7}{9}\right|\geq\epsilon\right)\leq e^{-Cn}.$$
\end{prop}

\begin{proof}
We know that, for each $n>1$, the sequence of random states $s_0,\ldots,s_{n-2}$ (from the first $n-1$ triples that act as input for the algorithm when started on $cut_n(T_\infty)^\circ$) has the distribution described in Proposition \ref{markovchain}. Since we have $d^\circ(cut_n(T_\infty))=|\{0\leq i<n| s_i\neq j\}|$, we also have $|d^\circ(cut_n(T_\infty))-|\{0\leq i\leq n-2| s_i\neq j\}||\leq 1$.

On the other hand,
$$\lim_{n\rightarrow\infty}\frac{1}{n}|\{0\leq i\leq n-2| s_i\neq j\}|$$
is a constant by the law of large numbers and is easily computed via the limit distribution for a Markov chain with transition matrix $M$, which is $\pi=\frac{1}{9}\left(1, 2, 4, 2\right)$.

As a consequence, we have
$$\lim_{n\rightarrow\infty} \frac{1}{n}d^\circ(cut_n(T_\infty))=\lim_{n\rightarrow\infty}\frac{1}{n}|\{0\leq i\leq n-2| s_i\neq j\}|=7/9.$$

The second part of the proposition is a direct consequence of classical results of Large Deviation Theory about Markov chains with a finite state space, see for example Chapter~3 of \cite{LDP}. The statement is true if we substitute $S_n=|\{0\leq i\leq n-2|s_i\neq j\}|$ for $d^\circ(cut_n(T_\infty))$, because $s_0,\ldots, s_{n-2}$ is a (recurrent) Markov chain with finite state space.

The inequality $|d^\circ(cut_n(T_\infty))-S_n|\leq 1$ implies that $\GW(d^\circ(cut_n(T_\infty))=d)\leq \GW(|S_n-d|\leq 1)$ for any natural number $d$, and thus
$$\GW\left(\left|\frac{d^\circ(cut_n(T_\infty))}{n}-\frac{7}{9}\right|\geq\epsilon\right)\leq \GW\left(\left|\frac{S_n}{n}-\frac{7}{9}\right|\geq\epsilon-\frac{1}{n}\right)$$
which establishes the result for $d^\circ(cut_n(T_\infty))$.
\end{proof}

\section{Final proofs}\label{conclusion}

The final technical steps follow \cite{CHKdissections}, but we need to bypass the rerooting argument exploited there by establishing a more practical control on all map-distances (not only map-distances from the root, as we have done so far). We will start, however, with a proposition dealing only with map-distances from the root, in order to extend the statement as soon as all of the necessary lemmas are in place.

Before we start, let us introduce a little notation: suppose $(x_n)_{n\geq 1}$ is a sequence of real numbers; we write $x_n=\oexp(n)$ to mean that $x_n\leq C_1e^{-C_2n^a}$ for some positive $C_1, C_2, a$. Also, in the following proposition and proof, we will write $c$ for the constant $\frac{7}{9}$, which we obtained in Section \ref{markov}, and $d$ for the usual graph distance on the outerplanar map obtained from a tree via the bijection $\Psi$, with $d(u)$ being the graph distance of $u$ from the root vertex. 

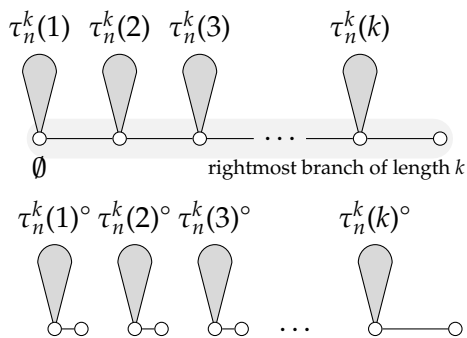
\begin{figure}
\centering
\begin{tikzpicture}[node distance=30 pt, b/.style={circle,inner sep=0pt,minimum size=5pt,fill=black,draw=black}, w/.style={circle,inner sep=0pt,minimum size=5pt, draw=black, fill=white}, g/.style={circle,inner sep=0pt,minimum size=5pt, draw=black, fill=gray}]
\node [w] (4) {};
\node [left of=4] (dots) {$\ldots$};
\node [w, right of=4] (i) {};
\node [w, left of=dots] (3) {};
\node [w, left of=3] (2) {};
\node [w, left of=2, label=below:$\root$] (1) {};
\draw (i)--(4)--(dots)--(3)--(2)--(1);
\begin{pgfonlayer}{background}\draw[gray!10, line width=15pt, line cap=round] (i)--(1) node[near start,below, color=black, scale=.7] {rightmost branch of length $k$};\end{pgfonlayer}
\begin{pgfonlayer}{background}
\foreach \x in {1,...,3} {
\node[ellipse callout, draw, inner sep=1pt, fill=gray!30!white, callout absolute pointer={(\x)}, above of=\x, yshift=-5pt, label=above:$\tau_n^k(\x)$,callout pointer arc=180] {\phantom{$T$}};}
\node[ellipse callout, draw, inner sep=1pt, fill=gray!30!white, callout absolute pointer={(4)}, above of=4, yshift=-5pt,label=above:$\tau_n^k(k)$, callout pointer arc=180] {\phantom{$T$}};
\end{pgfonlayer}
\end{tikzpicture}\\
\begin{tikzpicture}[node distance=30 pt, b/.style={circle,inner sep=0pt,minimum size=5pt,fill=black,draw=black}, w/.style={circle,inner sep=0pt,minimum size=5pt, draw=black, fill=white}, g/.style={circle,inner sep=0pt,minimum size=5pt, draw=black, fill=gray}]
\node [w] (4) {};
\node [left of=4] (dots) {$\ldots$};
\node [w, right of=4] (i) {};
\node [w, left of=dots] (3) {};
\node [w, left of=dots, xshift=10pt] (3a) {};
\node [w, left of=3] (2) {};
\node [w, left of=3,xshift=10pt] (2a) {};
\node [w, left of=2, xshift=10pt] (1a) {};
\node [w, left of=2] (1) {};
\draw (i)--(4) (3)--(3a) (2)--(2a) (1)--(1a);
\begin{pgfonlayer}{background}
\foreach \x in {1,...,3} {
\node[ellipse callout, draw, inner sep=1pt, fill=gray!30!white, callout absolute pointer={(\x)}, above of=\x, yshift=-5pt, label=above:$\tau_n^k(\x)^\circ$,callout pointer arc=180] {\phantom{$T$}};}
\node[ellipse callout, draw, inner sep=1pt, fill=gray!30!white, callout absolute pointer={(4)}, above of=4, yshift=-5pt,label=above:$\tau_n^k(k)^\circ$, callout pointer arc=180] {\phantom{$T$}};
\end{pgfonlayer}
\end{tikzpicture}
\caption{\label{rightmost branch} The tree $\tau_n^k$ seen as a forest of $k$ linked rooted plane trees.}
\end{figure}

\begin{prop}\label{final prop} Let $\tau_n$ be a random well bicoloured tree with $n$ vertices, and let $\di(\tau_n)$ be its (random) diameter; then for all $\epsilon>0$
$$p_n=\GW\left(\exists u \in \tau_n \mbox{ s.t. } |d(\tau_n,u)-c|u||\geq\epsilon\max\{\di(\tau_n), \sqrt{n}\}\right)=\oexp(n).$$
\end{prop}

\begin{proof} For any positive integer $k$, we call $\tau_n^k$ a random well bicoloured tree with $n$ vertices conditioned on having a rightmost branch of length exactly $k$. We claim that the probability of $\tau_n$ having a rightmost branch of length greater than $\epsilon\sqrt{n}/4$ is $\oexp(n)$; we shall need a more general result before the end of this section, and we postpone the proof of this claim, in a stronger form, until the end of the proof (see Lemma \ref{ancestor path}). Given the claim, we have that $p_n$ is $\oexp(n)$ if and only if the same is true for

$$\sum_{k=1}^{\frac{\epsilon}{4}\sqrt{n}} q_{n,k}\GW \left(\exists u \in \tau^k_n \mbox{ s.t. } |d(\tau^k_n,u)-c|u||\geq\epsilon\max\{\di(\tau^k_n), \sqrt{n}\}\right)$$
where $q_{n,k}$ is the probability that $\tau_n$ has a rightmost branch of length exactly $k$. If we were able to prove that each probability in the sum is actually $\oexp(n)$, then we would have $p_n=\oexp(n)$ as well.

Let us consider a single term of the sum. Notice that, for each $k$, $\tau_n^k$ can be seen as a random forest of $k$ ordered trees, with $n-1$ vertices between them, such that the trees are linked by the roots with a path going from left to right, and a single extra vertex is linked to the root of the rightmost tree; each of the trees is bicoloured (not necessarily \emph{well} bicoloured) with the only condition of having a white root; the rightmost vertex is white. 


We label the $k$ (random) trees in the forest $\tau_n^k$, ordered from left to right, $\tau_n^k(1),\ldots,\tau_n^k(k)$; if, for each $i$ between $1$ and $k$, we condition $\tau_n^k(i)$ on having a certain number $n_i$ of vertices (with $n_i\geq 1$ and $n_1+\ldots+n_k=n-1$), then $\tau_n^k(i)$ simply becomes a random plane tree with $n_i$ vertices (uniformly bicoloured but with a white root), which we call $\theta_{n_i}$.

Notice now that, for any vertex $u$ in $\tau_n^k(i)$, $|d(\tau_n^k,u)-d^\circ(\tau_n^k(i),u)|\leq k$; also, the height of $u$ in $\tau_n^k(i)$ differs from the height of $u$ in $\tau_n^k$ by at most $k$, and we have the obvious inequality $\di(\tau_n^k(i))\leq \di(\tau_n^k)$ between diameters. As a consequence,
$$\GW \left(\exists u \in \tau^k_n \mbox{ s.t. } |d(\tau^k_n,u)-c|u||\geq\epsilon\max\{\di(\tau^k_n), \sqrt{n}\}\right)\leq$$
$$\leq \sum_{i=1}^k \GW \left(\exists u \in \tau^k_n(i) \mbox{ s.t. } |d^\circ(\tau^k_n(i),u)-c|u||\geq\epsilon\max\{\di(\tau^k_n(i)), \sqrt{n}\}-2k\right)$$
where in the second expression we still write $|u|$ for the height of the vertex in $\tau_n^k(i)$. 

Hence the probability above is no more than
$$\max_{\substack{n_1,\ldots,n_k>0\\n_1+\ldots+n_k=n-1}}\sum_{i=1}^{k}\GW\left(\exists u \in \theta_{n_i} \mbox{ s.t. } |d^\circ(\theta_{n_i},u)-c|u||\geq\epsilon\max\{\di(\theta_{n_i}), \sqrt{n}\}-2k\right).$$

Now, we know that $k\leq \frac{\epsilon}{4}\sqrt{n}$, and therefore we can reduce to evaluating the expression
$$\frac{\epsilon}{4}\sqrt{n}\max_{0<m<n}\GW\left(\exists u \in \theta_{m} \mbox{ s.t. } |d^\circ(\theta_{m},u)-c|u||\geq\frac{\epsilon}{2}\max\{\di(\theta_{m}), \sqrt{n}\}\right).$$

Let $\theta$ be a uniformly bicoloured Galton-Watson tree (see Section \ref{infinite tree}); the probability that $\theta$ has $m$ vertices is $2^{-2m-1}Cat(m-1)$, which is asymptotic to $m^{-\frac{3}{2}}$; we can therefore find a constant $C_1$ such that, for any $m$, $2^{-2m-1}Cat(m-1)\geq C_1m^{-\frac{3}{2}}$. This guarantees that, for any positive integer $m$ less than $n$,
$$\GW\left(\exists u \in \theta_{m} \mbox{ s.t. } |d^\circ(\theta_{m},u)-c|u||\geq\frac{\epsilon}{2}\max\{\di(\theta_{m}), \sqrt{n}\}\right) \leq$$ 
$$\leq \frac{1}{C_1}m^{\frac{3}{2}}\cdot\GW\left(\exists u \in \theta \mbox{ s.t. } |d^\circ(\theta,u)-c|u||\geq\frac{\epsilon}{2}\max\{\di(\theta), \sqrt{n}\}\right).$$

If we write $q_n$ for the probability appearing in the above expression, then (since $m\leq n$) proving $q_n=\oexp(n)$ would guarantee that the above -- and therefore $p_n$ -- is $\oexp(n)$. We shall now turn to the former endeavour.

It is clear that
$$q_n\leq\GW\left[\sum_{u\in \theta}\1\left(|d^\circ(\theta,u)-c|u||\geq\frac{\epsilon}{2}\max\{\di(\theta), \sqrt{n}\}\right)\right]$$
and, grouping vertices together according to height, the latter can be rewritten as
$$\GW\left[\mbox{\fontsize{10}{30}\selectfont $\displaystyle \sum_{j\geq 1}\sum_{u\in \theta\atop|u|=j}\1\left(|d^\circ(\theta,u)-cj|\geq\frac{\epsilon}{2}\max\{\di(\theta), \sqrt{n}\}\right)$}\right]\leq\sum_{j\geq 1}\GW\left[\mbox{\fontsize{10}{30}\selectfont $\displaystyle\sum_{u\in \theta\atop|u|=j}\1\left(|d^\circ(\theta,u)-cj|\geq\frac{\epsilon}{2}\max\{j, \sqrt{n}\}\right)$}\right]$$
where we have used the fact that, for each $u$ in the sum, we have $d(\theta)\geq|u|$.

We can now use Lemma \ref{go infinite} to turn the expression above into
$$\sum_{j\geq 1}\GW\left(|d^\circ(cut_j(T_\infty))-cj|\geq\frac{\epsilon}{2}\max\{j, \sqrt{n}\}\right)$$
where $T_\infty$ is the critical geometric Galton-Watson tree conditioned to survive, randomly bicoloured, as presented in Section \ref{infinite tree}.

We split the sum into two parts, which we will deal with separately: the sum for $j\leq n^{1/4}$ and that of the terms with $j>n^{1/4}$.

Suppose that $j\leq n^{1/4}$; clearly, $d^\circ(cut_j(T_\infty))\leq j$, and $\max\{j, \sqrt{n}\}=\sqrt{n}$. This gives $|d^\circ(cut_j(T_\infty))-cj|\leq d^\circ(cut_j(T_\infty))+cj\leq (1+c)j\leq(1+c)n^{1/4}$. Thus it is enough to choose $n$ so that $(1+c)n^{1/4}<\frac{\epsilon}{2}\sqrt{n}$, that is $n^{1/4}>\frac{2+2c}{\epsilon}$, to obtain that
$$\sum_{j=1}^{\lfloor n^{1/4}\rfloor}\GW\left(|d^\circ(cut_j(T_\infty))-cj|\geq\frac{\epsilon}{2}\max\{j, \sqrt{n}\}\right)=0.$$

As for the sum with $j> n^{1/4}$, we have
$$\GW\left(\left|{d^\circ(cut_j(T_\infty))}-cj\right|\geq\frac{\epsilon}{2}\max\{j, \sqrt{n}\}\right)\leq\mathbb{P}\left(\left|\frac{d^\circ(cut_j(T_\infty))}{j}-c\right|\geq\frac{\epsilon}{2}\right)$$
which, by choosing $n$ appropriately according to Proposition \ref{markovchainestimate}, can be bounded by $e^{-Cn}$.

This gives, for $n$ suitably big, the bound $$\sum_{j>n^{1/4}}e^{-Cj}=\oexp(n),$$
 which is our aim.

\end{proof}

We now prove the claim from the beginning of the proof, in the form of the following Lemma:

\begin{lem}\label{ancestor path}Let $\tau_n$ be a random well bicoloured plane tree with $n$ vertices, and let $\alpha$ and $\epsilon$ be positive real numbers; we call a path in $\tau_n$ an \emph{ancestor path} if it is of the form $x_0\ldots x_l$, with $x_{i}=p(x_{i-1})$ for all $i$ between 1 and $l$. Then
$$\GW\left(\mbox{$\tau_n$ has an ancestor path of length at least $\alpha n^\epsilon$ entirely made up of white vertices}\right) = \oexp(n);$$
in particular, the probability that $\tau_n$ has a rightmost branch of length at least $\alpha n^\epsilon$ is $\oexp(n)$.
\end{lem}

\begin{proof}We start by showing the last, more specific assertion: that the probability of $\tau_n$ having a rightmost branch of length at least $\alpha n^\epsilon$ is $\oexp(n)$.

For any $d>0$, the number of bicoloured trees with $n$ vertices and a rightmost branch of length $d$ (as seen for example in \cite{bijectionpaper}) is
$$2^{n-1-d}Cat(n-1-d,d)=2^{n-1-d}\frac{d}{2n-2-d}{2n-2-d \choose n-1-d},$$
that is the number of plane trees with $n$ vertices and a rightmost branch of length $d$ (or, equivalently, of sequences of $d$ plane trees with $n-1$ vertices in total, see Figure~\ref{rightmost branch}), multiplied by the number of possible bicolourings ($2^{n-1-d}$, since $d+1$ out of the $n$ vertices belong to the rightmost branch of the tree).

The total number of outerplanar maps with $n$ vertices is asymptotic (up to a multiplicative constant) to $2^{3n}n^{-\frac{3}{2}}$, as can be easily obtained from Stirling estimates for the above formula (see, for detailed analogous computations, \cite{bijectionpaper}); on the other hand,
$$\sum_{\alpha n^{\epsilon}<d<n}2^{n-1-d}Cat(n-1-d,d)\leq \sum_{\alpha n^{\epsilon}<d<n} 2^{n-1-d}{2n-2-d \choose n-1-d},$$
which in turn is less than
$$2^{3n}p(n)\sum_{\alpha n^{\epsilon}<d<n}2^{-2d},$$
where $p(n)$ a polynomial in $n$. As a consequence, the above expression divided by the total number of outerplanar maps with $n$ vertices is $\oexp(n)$. 

It is now very easy to extend the result to general ancestor paths. Since the probability that $\tau_n$ has a rightmost branch of length at least $\frac{\alpha }{2}n^\epsilon$ we have shown to be $\oexp(n)$, we may assume $\tau_n$ is conditioned on having a rightmost branch of length less than $\frac{\alpha }{2}n^\epsilon$. For each vertex $v$ of height at least $\alpha n^{\epsilon}$ in $\tau_n$ consider the path $P_v=vp(v)\ldots p^i(v)\ldots p^{\lceil \alpha  n^\epsilon\rceil}(v)$. No more than $\frac{1}{2}\alpha n^\epsilon$ of its vertices belong to the rightmost branch of $\tau_n$, and thus the probability of $P_v$ being entirely white is at most $2^{-\frac{\alpha }{2}n^\epsilon}$. Hence
$$\GW\left(\ \parbox{11cm}{\centering$\tau_n$ has an ancestor path of length $\alpha n^\epsilon$ entirely made up of white vertices and the rightmost branch of $\tau_n$ is shorter than $\frac{\alpha }{2}n^\epsilon$}\ \right)<n\cdot2^{-\frac{\alpha }{2}n^\epsilon}$$
which is $\oexp(n)$ as wanted.
\end{proof}

Proposition \ref{final prop} is a substantial step toward being able to bound the Gromov-Hausdorff distance between a (rescaled) tree and its corresponding planar map, but dealing with distances from the root is not enough: we need a way to derive results of the same kind about distances between generic vertices.

To this end, we will rewrite the map-distance between two vertices in terms of the distances between each vertex and the root. This is easily done when the vertices in question are related, and the one of smaller height is black; this basic case we will use as a stepping stone, together with Lemma \ref{ancestor path}, to establish the required general result.

\begin{lem}\label{ancestor}Let $\tau$ be a bicoloured tree and $v$ a vertex in $\tau$; let $w$ be a black ancestor of $v$ in $\tau$, $\emptyset$ the root of $\tau$; call $d_M$ the map-distance on $\tau$ and (as in Section \ref{preliminary lemmas}) write $d_M(u)$ for $d_M(u,\emptyset)$. Then $|d_M(v,w)-d_M(v)+d_M(w)|\leq 2$.\end{lem}
\hspace{-21pt}
\begin{tabular}{@{}p{0.85\textwidth}@{ }p{0.1\textwidth}@{}}
\vspace{-14pt}
\begin{proof}
Consider a map-geodesic from $v$ to the root; if this path goes through $w$, then $d_M(v)=d_M(v,w)+d_M(w)$. If it does not, then at some point it jumps from a descendant $s$ of $w$ onto a target $t$ `below' $w$, which must be the child of an ancestor of $w$, thus also the target of $w$. We have
$$d_M(v,s)\leq d_M(v,w)\leq d_M(v,s)+2$$
$$d_M(v)=d_M(v,s)+1+d_M(t)$$
$$d_M(t)\leq d_M(w)\leq d_M(t)+1$$
hence $|d_M(v)-d_M(v,w)-d_M(w)|\leq 2$ as wanted.
\end{proof} &
\vspace{5pt}
\begin{tikzpicture}[node distance=20 pt, v/.style={circle,inner sep=0pt,minimum size=4pt,fill=black}, g/.style={circle,inner sep=0pt,minimum size=4pt,fill=gray}, c/.style={circle,inner sep=0pt,minimum size=4pt, color=black, fill=red}]
\node [v, xshift=50pt, yshift=30pt, label=left:$w$] (w) {};
\node [g, above of=w, yshift=20pt, label=$v$] (v) {};
\node [v, above right of=w, label=right:$s$] (s) {};
\draw [decorate, decoration={random steps,segment length=4pt,amplitude=2pt},
              rounded corners=1pt, color=red] (v) -- +(4pt,3pt) -- (s);
\node [g, xshift=53pt,yshift=50pt] (1) {};
\node [g, below right of=w, label=right:$t$] (t) {};
\node [g, below of=w] (2) {};
\node [c, below of=2,yshift=-12pt] (x) {};
\draw [decorate, decoration={zigzag,segment length=7,amplitude=.9,
  post=lineto,post length=2pt}] (2)--(w)--(1)--(v);
\draw [draw=black] (t)--(2);
\draw [decorate, decoration={zigzag,segment length=7,amplitude=.9,
  post=lineto,post length=2pt}] (s)--(1);
\draw [decorate, decoration={zigzag,segment length=7,amplitude=.9,
  post=lineto,post length=2pt}] (2)-- (x);
\draw [densely dashed, ->, color=red] (s) .. controls +(south east:10pt) and +(north east:10pt) .. (t);
\draw [densely dashed, ->] (w) -- (t);
\draw [decorate, decoration={random steps,segment length=2pt,amplitude=2pt},
              rounded corners=1pt, color=red] (t) -- (x);
\end{tikzpicture}\\
\end{tabular}

Here is a general statement analogous to Proposition \ref{final prop}, where we write $d_T(u,v)$ for the distance of two vertices \emph{in the tree}, and $d_M$ for the map distance (again, with $d_M(u)$ being the map distance from the root).

\begin{cor}\label{black distances}
 Let $\tau_n$ be a random well bicoloured tree with $n$ vertices, and let $\di(\tau_n)$ be its (random) diameter; then for all $\epsilon>0$
$$\GW\left(\exists u, v \in \tau_n \mbox{ s.t. } |d_M(u,v)-cd_T(u,v)|\geq\epsilon\max\{\di(\tau_n), \sqrt{n}\}\right)=\oexp(n).$$
\end{cor} 

\begin{proof}
Thanks to Lemma \ref{ancestor path} (by choosing $\alpha=\frac{\epsilon}{8}$) we may restrict ourselves to the event of $\tau_n$ having no white ancestor path of length $\frac{\epsilon}{8}\sqrt{n}$ or greater.

Consider, given $u$ and $v$ vertices of $\tau_n$, their first common ancestor $w$. Either $|w|>\frac{\epsilon}{8}\sqrt{n}$, in which case there is a \emph{black} ancestor $z$ of $w$ such that $d_T(w,z)\leq\frac{\epsilon}{8}\sqrt{n}$ (otherwise there would be a long white ancestor path), or $|w|<\frac{\epsilon}{8}\sqrt{n}$, in which case we just set $z$ to be the root of $\tau_n$.

Suppose without loss of generality that $u$ lies to the left of $v$ and let $w'$ be the child of $w$ that is also an ancestor of $v$; then a map-geodesic from $u$ to $v$ goes through at least one of $w$ and $w'$, thanks to an argument very similar to that employed in Proposition \ref{vlemma}.
This yields that $|d_M(u,v)-d_M(u,w)-d_M(w,v)|\leq 2$ (this is trivial if the map-geodesic passes through $w$; if it goes through $w'$ then we have $|d_M(u,v)-d_M(u,w)-d_M(w,v)|=|d_M(u,w')+d_M(w',v)-d_M(u,w)-d_M(w,v)|\leq|d_M(u,w')-d_M(u,w)|+|d_M(w',v)-d_M(w,v)|\leq2$).

Also notice that we have $|d_M(u,w)-d_M(u,z)|\leq d_M(w,z)\leq \frac{\epsilon}{8}\sqrt{n}$, and the same inequality is true if we substitute $v$ for $u$.

Now, if $z$ is black, then $|d_M(u,z)-d_M(u)+d_M(z)|\leq 2$ and $|d_M(v,z)-d_M(v)+d_M(z)|\leq 2$ by Lemma \ref{ancestor}; otherwise $z$ is the root of $\tau_n$, and the same assertions are trivial (since $d_M(u,z)=d_M(u)$ and $d_M(z)=0$).

All of the above observations combined yield
$$|d_M(u,v)-cd_T(u,v)|=|d_M(u,v)-c(|u|+|v|-2|w|)|\leq$$
$$\leq\frac{\epsilon}{2}\sqrt{n}+6+|d_M(u)-c|u||+|d_M(v)-c|v||+2|d_M(z)-c|z||.$$

Hence
$$\GW\left(\exists u, v \in \tau_n \mbox{ s.t. } |d_M(u,v)-cd_T(u,v)|\geq\epsilon\max\{\di(\tau_n), \sqrt{n}\}\right)\leq$$
$$\GW\left(\exists u, v, z \in \tau_n \mbox{ s.t. }\mbox{\fontsize{10}{30}\selectfont $\frac{\epsilon}{2}\sqrt{n}+6+|d_M(u)-c|u||+|d_M(v)-c|v||+2|d_M(z)-c|z||\geq\epsilon\max\{\di(\tau_n), \sqrt{n}\}$}\right)\leq$$
$$\GW\left(\exists u, v, z \in \tau_n \mbox{ s.t. } |d_M(u)-c|u||+|d_M(v)-c|v||+2|d_M(z)-c|z||\geq\frac{\epsilon}{4}\max\{\di(\tau_n), \sqrt{n}\}\right)\leq$$
$$\GW\left(\exists u\in \tau_n \mbox{ s.t. } |d_M(u)-c|u||\geq\frac{\epsilon}{16}\max\{\di(\tau_n), \sqrt{n}\}\right)$$
which we show to be $\oexp(n)$ by invoking Proposition \ref{final prop}.
\end{proof}

The time has come for the proof of our main theorem, which is now quite straightforward; we restate it here:

{
\renewcommand{\theteo}{1.1}
\begin{teo}Let $\mathbf{M}_{n}$ be a random uniform rooted simple outerplanar map with $n$ vertices, and denote by $ d_{gr}$ the graph distance on  the set of its vertices $V(\mathbf{M}_{n})$. We have the following convergence in distribution for the Gromov-Hausdorff topology:
$$ \left( V(\mathbf{M}_{n}), \frac{d_{gr}}{ \sqrt{n}} \right) \quad \xrightarrow[n\to\infty]{(d)} \quad \frac{7 \sqrt{2}}{9} \cdot ( \mathcal{T}_{e},d),$$
where $ ( \mathcal{T}_{e}, d)$ is the Brownian CRT of Aldous. We adopt here the normalisation of Le Gall \cite{GW} by considering  $ \mathcal{T}_{e}$ as constructed from a normalised Brownian excursion.\end{teo}
}

\begin{proof}
Corollary \ref{black distances} yields that, given a random well bicoloured tree with $n$ vertices $\tau_n$,
$$\GW\left(d_{GH}(\Psi(\tau_n), c\tau_n)\geq \epsilon\max\{\sqrt{n},\di(\tau_n)\} \right)=\oexp(n)$$
where $\Psi(\tau_n)$ in this context is seen as the metric space made up of the vertices of $\tau$ equipped with the map-distance, and $c\tau_n$ is the set of vertices in $\tau_n$, equipped with the graph distance of $\tau_n$ rescaled by a factor $c=\frac{7}{9}$.

This is because
$$d_{GH}(\Psi(\tau_n), c\tau_n)\leq\frac{1}{2}\left(\sup_{u,v\in \tau_n}|d_M(u,v)-cd_T(u,v)|\right)\!,$$
as seen by considering the trivial correspondence between the vertices of $\tau_n$ and those of $\Psi(\tau_n)$.

Thus we have established that
$$\lim_{n\rightarrow\infty}d_{GH}\left(\frac{\Psi(\tau_n)}{\max\{\sqrt{n},\di(\tau_n)\}},\frac{c\tau_n}{\max\{\sqrt{n},\di(\tau_n)\}}\right)=0$$
in probability, where $\tau_n$ is a random rooted well bicoloured plane tree with $n$ vertices.

We now claim that
$$\lim_{n\rightarrow\infty}\left(\tau_n, \frac{d_{gr}}{\sqrt{2n}}\right)=(\mathcal{T}_{\mathbf{e}},d)$$
in distribution for the Gromov-Hausdorff distance, with $(\mathcal{T}_{\mathbf{e}},d)$ being the CRT. 

This result is a consequence of a famous theorem of Aldous \cite{Ald93} and would be immediate if $\tau_n$ were replaced by a uniform plane tree with $n$ vertices or, equivalently, by a critical geometric Galton-Watson tree conditioned on having $n$ vertices. Even though this is not the case, $\tau_n$ is not very far from the latter: indeed, we saw in Lemma \ref{ancestor path} that the length $L_n$ of the rightmost branch of $\tau_n$ remains tight (it even converges in distribution) as $n\rightarrow \infty$, and furthermore that, conditionally on ${L_n = k}$, the $k$ subtrees grafted onto the rightmost branch form a forest $\tau_n^k(1), \tau_n^k(2), \ldots, \tau_n^k(k)$ whose total number of vertices is $n-1$. It is known that such a forest has, as $n\rightarrow\infty$, a unique macroscopic tree $\tilde{\tau}_n$ of size $s_n=n-o(n)$, which is uniformly distributed over all plane trees of size $s_n$. The scaling limit of $\tau_n$ is thus the same as that of $\tilde{\tau}_n$, which is that of a random uniform plane tree of size $n$. See Section~3.3 of \cite{MP02} for details.

But then the random variable $$\frac{\max\{\sqrt{n},\di(\tau_n)\}}{\sqrt{n}}$$ also converges in distribution, and to an almost surely positive random variable. That is, if we multiply by $\frac{\max\{\sqrt{n},d(\tau_n)\}}{\sqrt{n}}$ we can in fact deduce that
$$\lim_{n\rightarrow\infty}d_{GH}\left(\frac{\Psi(\tau_n)}{\sqrt{n}},\frac{c\tau_n}{\sqrt{n}}\right)=0.$$

Now remember that, thanks to Theorem \ref{bij}, $\Psi(\tau_n)$ is a random rooted simple outerplanar map with $n$ vertices, that is it has the same distribution as $\mathbf{M}_n$. This finally gives
$$\lim_{n\rightarrow\infty}\left(\mathbf{M}_n, \frac{d_{gr}}{\sqrt{n}}\right)=\lim_{n\rightarrow\infty}\left(\tau_n, \frac{7d_{gr}}{9\sqrt{n}}\right)=\left(\mathcal{T}_{\mathbf{e}},\frac{7\sqrt{2}}{9}d\right).$$
\end{proof}

\section*{Acknowledgements}
I would like to thank N. Curien for his numerous suggestions and improvements to the readability of this paper; J-F. Le Gall and F. Flandoli for their supervision and support; finally, D. Lombardo for the many long and in-depth discussions about the subject matter.

\bibliographystyle{acm}
\bibliography{outerplanarbib}

\begin{thebibliography}{10}

\bibitem{ABBG12}
{\sc Addario-Berry, L., Broutin, N., and Goldschmidt, C.}
\newblock The continuum limit of critical random graphs.
\newblock {\em Probab. Theory Relat. Fields 152}, 3-4 (2012), 367--406.

\bibitem{AM08}
{\sc Albenque, M., and Marckert, J.-F.}
\newblock Some families of increasing planar maps.
\newblock {\em Electron. J. Probab. 13\/} (2008), no. 56, 1624--1671.

\bibitem{Ald91a}
{\sc Aldous, D.}
\newblock The continuum random tree. {I}.
\newblock {\em Ann. Probab. 19}, 1 (1991), 1--28.

\bibitem{Ald91}
{\sc Aldous, D.}
\newblock {\em The Continuum random tree II: an overview}.
\newblock London Mathematical Society Lecture Note Series. Cambridge University
  Press, 1991.

\bibitem{Ald93}
{\sc Aldous, D.}
\newblock The continuum random tree. {III}.
\newblock {\em Ann. Probab. 21}, 1 (1993), 248--289.

\bibitem{Bet11}
{\sc {Bettinelli}, J.}
\newblock {Scaling Limit of Random Planar Quadrangulations with a Boundary}.
\newblock {\em Annales de l'institut Henri Poincar\'e\/} (2011), to appear.

\bibitem{bijectionpaper}
{\sc Bonichon, N., Gavoille, C., and Hanusse, N.}
\newblock Canonical decomposition of outerplanar maps and application to
  enumeration, coding and generation.
\newblock {\em J. Graph Algorithms Appl. 9}, 2 (2005), 185--204.

\bibitem{BG09}
{\sc Bouttier, J., and Guitter, E.}
\newblock Distance statistics in quadrangulations with a boundary, or with a
  self-avoiding loop.
\newblock {\em J. Phys. A 42}, 46 (2009), 465208, 44.

\bibitem{minors}
{\sc Chartrand, G., and Harary, F.}
\newblock Planar permutation graphs.
\newblock {\em Annales de l'institut Henri Poincar\'e (B) Probabilit\'es et
  Statistiques 3}, 4 (1967), 433--438.

\bibitem{CHKdissections}
{\sc Curien, N., Haas, B., and Kortchemski, I.}
\newblock {The CRT is the scaling limit of random dissections}.
\newblock {\em Random Struct. Algorithms\/}, to appear.

\bibitem{LDP}
{\sc Dembo, A., and Zeitouni, O.}
\newblock {\em Large deviations techniques and applications}, vol.~38 of {\em
  Stochastic Modelling and Applied Probability}.
\newblock Springer-Verlag, Berlin, 2010.

\bibitem{HM12}
{\sc Haas, B., and Miermont, G.}
\newblock Scaling limits of {M}arkov branching trees with applications to
  {G}alton-{W}atson and random unordered trees.
\newblock {\em Ann. Probab. 40}, 6 (2012), 2589--2666.

\bibitem{JS12}
{\sc {Janson}, S., and {Orn Stefansson}, S.}
\newblock {Scaling limits of random planar maps with a unique large face}.
\newblock {\em Ann. Probab.\/}, to appear.

\bibitem{GW}
{\sc Le~Gall, J.-F.}
\newblock Random trees and applications.
\newblock {\em Probab. Surv. 2\/} (2005), 245--311.

\bibitem{LG11}
{\sc Le~Gall, J.-F.}
\newblock Uniqueness and universality of the {B}rownian map.
\newblock {\em Ann. Probab. 41}, 4 (2013), 2880--2960.

\bibitem{lyons}
{\sc Lyons, R., and Peres, Y.}
\newblock Probability on trees and networks, 2005.

\bibitem{MP02}
{\sc Marckert, J.-F., and Panholzer, A.}
\newblock Noncrossing trees are almost conditioned {G}alton-{W}atson trees.
\newblock {\em Random Struct. Algorithms 20}, 1 (2002), 115--125.

\bibitem{Mie11}
{\sc Miermont, G.}
\newblock The {B}rownian map is the scaling limit of uniform random plane
  quadrangulations.
\newblock {\em Acta Math. 210}, 2 (2013), 319--401.

\bibitem{PSW}
{\sc Panagiotou, K., Stufler, B., and Weller, K.}
\newblock Scaling limits of random graphs from subcritical classes, in
  preparation.

\bibitem{char}
{\sc Sys{\l}o, M.~M.}
\newblock Characterizations of outerplanar graphs.
\newblock {\em Discrete Math. 26}, 1 (1979), 47--53.

\end{thebibliography}

\end{document}